\documentclass{amsart}
\usepackage{amsthm,amssymb,enumitem}
\usepackage{tikz}

\usetikzlibrary{matrix,arrows}
\tikzstyle{dot}=[circle, draw=black, fill=black!25, inner sep=.4ex]
\tikzstyle{whitedot}=[circle, draw=black, fill=white, inner sep=.4ex]
\tikzstyle{blackdot}=[circle, draw=black, fill=black, inner sep=.4ex]
\newif\ifvflip\pgfkeys{/tikz/vflip/.is if=vflip}
\newif\ifhflip\pgfkeys{/tikz/hflip/.is if=hflip}
\newif\ifhvflip\pgfkeys{/tikz/hvflip/.is if=hvflip}
\newlength\morphismheight
\newlength\wedgewidth
\tikzset{width/.initial=1mm}
\makeatletter
\tikzstyle{morphism}=[font=\small,morphismshape]
\pgfdeclareshape{morphismshape}
{
    \savedanchor\centerpoint
    {
        \pgf@x=0pt
        \pgf@y=0pt
    }
    \anchor{center}{\centerpoint}
    \anchorborder{\centerpoint}
    \saveddimen\overallwidth
    {
        \pgfkeysgetvalue{/tikz/width}{\minwidth}
        \pgf@x=\wd\pgfnodeparttextbox
        \ifdim\pgf@x<\minwidth
            \pgf@x=\minwidth
        \fi
    }
    \savedanchor{\upperrightcorner}
    {
        \pgf@y=.5\ht\pgfnodeparttextbox
        \advance\pgf@y by -.5\dp\pgfnodeparttextbox
        \pgf@x=.5\wd\pgfnodeparttextbox
    }
    \anchor{north}
    {
        \pgf@x=0pt
        \pgf@y=0.5\morphismheight
    }
    \anchor{north east}
    {
        \pgf@x=\overallwidth
        \multiply \pgf@x by 2
        \divide \pgf@x by 5
        \pgf@y=0.5\morphismheight
    }
    \anchor{east}
    {
        \pgf@x=\overallwidth
        \divide \pgf@x by 2
        \advance \pgf@x by 5pt
        \pgf@y=0pt
    }
    \anchor{west}
    {
        \pgf@x=-\overallwidth
        \divide \pgf@x by 2
        \advance \pgf@x by -5pt
        \pgf@y=0pt
    }
    \anchor{north west}
    {
        \pgf@x=-\overallwidth
        \multiply \pgf@x by 2
        \divide \pgf@x by 5
        \pgf@y=0.5\morphismheight
    }
    \anchor{south east}
    {
        \pgf@x=\overallwidth
        \multiply \pgf@x by 2
        \divide \pgf@x by 5
        \pgf@y=-0.5\morphismheight
    }
    \anchor{south west}
    {
        \pgf@x=-\overallwidth
        \multiply \pgf@x by 2
        \divide \pgf@x by 5
        \pgf@y=-0.5\morphismheight
    }
    \anchor{south}
    {
        \pgf@x=0pt
        \pgf@y=-0.5\morphismheight
    }
    \anchor{text}
    {
        \upperrightcorner
        \pgf@x=-\pgf@x
        \pgf@y=-\pgf@y
    }
    \backgroundpath
    {
    \begin{scope}
        \pgfkeysgetvalue{/tikz/fill}{\morphismfill}
        \pgfsetstrokecolor{black}
        \pgfsetlinewidth{.7pt}
        \begin{scope}
        \pgfsetstrokecolor{black}
        \pgfsetfillcolor{white}
        \pgfsetlinewidth{.7pt}
                \ifhflip
                    \pgftransformyscale{-1}
                \fi
                \ifvflip
                    \pgftransformxscale{-1}
                \fi
                \ifhvflip
                    \pgftransformxscale{-1}
                    \pgftransformyscale{-1}
                \fi
                \pgfpathmoveto{\pgfpoint
                    {-0.5*\overallwidth-5pt}
                    {0.5*\morphismheight}}
                \pgfpathlineto{\pgfpoint
                    {0.5*\overallwidth+5pt}
                    {0.5*\morphismheight}}
                \pgfpathlineto{\pgfpoint
                    {0.5*\overallwidth + \wedgewidth}
                    {-0.5*\morphismheight}}
                \pgfpathlineto{\pgfpoint
                    {-0.5*\overallwidth-5pt}
                    {-0.5*\morphismheight}}
                \pgfpathclose
                \pgfusepath{fill,stroke}
        \end{scope}
    \end{scope}
    }
}
\makeatother
\newcommand{\tinymult}[1][dot]{
\smash{\raisebox{-2pt}{\hspace{-5pt}\ensuremath{\begin{pic}[scale=0.4,yscale=-1]
    \node (0) at (0,0) {};
    \node[#1, inner sep=1.5pt] (1) at (0,0.55) {};
    \node (2) at (-0.5,1) {};
    \node (3) at (0.5,1) {};
    \draw (0.center) to (1.center);
    \draw (1.center) to [out=left, in=down, out looseness=1.5] (2.center);
    \draw (1.center) to [out=right, in=down, out looseness=1.5] (3.center);
    \node[#1, inner sep=1.5pt] (1) at (0,0.55) {};
\end{pic}
}\hspace{-3pt}}}}
\newcommand{\tinyunit}[1][dot]{
\smash{\raisebox{1pt}{\hspace{-3pt}\ensuremath{\begin{pic}[scale=0.4,yscale=-1]
    \node (0) at (0,0) {};
    \node[#1, inner sep=1.5pt] (1) at (0,0.55) {};
    \draw (0.center) to (1.north);
\end{pic}
}\hspace{-1pt}}}}

\newenvironment{pic}[1][]
{\begin{aligned}\begin{tikzpicture}[font=\tiny,#1]}
{\end{tikzpicture}\end{aligned}}

\newcommand{\inprod}[2]{\langle #1 \mid #2 \rangle}
\renewcommand{\L}{\ensuremath{\mathcal{L}}}
\newcommand{\cat}[1]{\ensuremath{\mathbf{#1}}}
\newcommand{\op}{\ensuremath{^{\mathrm{op}}}}
\newcommand{\cp}{\ensuremath{\mathrm{cp}}}
\newcommand{\bd}{\ensuremath{^{\mathrm{bd}}}}
\newcommand{\id}[1][]{\ensuremath{\mathrm{id}_{#1}}}
\newcommand{\C}{\ensuremath{\mathbb{C}}}
\newcommand{\M}{\ensuremath{\mathbb{M}}}
\newcommand{\N}{\ensuremath{\mathbb{N}}}
\DeclareMathOperator{\Loc}{Loc}
\DeclareMathOperator{\diag}{diag}
\DeclareMathOperator{\CPs}{CP}
\DeclareMathOperator{\tr}{tr}
\DeclareMathOperator{\supp}{supp}
\DeclareMathOperator{\ev}{ev}

\DeclareMathOperator{\Radon}{Radon}

\DeclareMathOperator{\Spec}{Spec}
\theoremstyle{plain}
\newtheorem{theorem}{Theorem}[section]
\newtheorem{proposition}[theorem]{Proposition}
\newtheorem{corollary}[theorem]{Corollary}
\newtheorem{lemma}[theorem]{Lemma}

\theoremstyle{definition}
\newtheorem{definition}[theorem]{Definition}
\newtheorem{example}[theorem]{Example}
\newtheorem{remark}[theorem]{Remark}

\begin{document}
\title{Frobenius structures over Hilbert C*-modules}
\author{Chris Heunen and Manuel L. Reyes}
\date{\today. We thank Andreas Blass, Bertfried Fauser, Simon Henry, Klaus Keimel, and Sean Tull, and gratefully acknowledge support by EPSRC Fellowship EP/L002388/1 and NSF grant DMS-1407152.
}
\begin{abstract}
  We study the monoidal dagger category of Hilbert C*-modules over a commutative C*-algebra from the perspective of categorical quantum mechanics.
  The dual objects are the finitely presented projective Hilbert C*-modules.
  Special dagger Frobenius structures correspond to bundles of uniformly finite-dimensional C*-algebras.
  A monoid is dagger Frobenius over the base if and only if it is dagger Frobenius over its centre and the centre is dagger Frobenius over the base. We characterise the commutative dagger Frobenius structures as finite coverings, and give nontrivial examples of both commutative and central dagger Frobenius structures. 
  Subobjects of the tensor unit correspond to clopen subsets of the Gelfand spectrum of the C*-algebra, and we discuss dagger kernels.
\end{abstract}
\maketitle

\section{Introduction}

\emph{Categorical quantum mechanics}~\cite{heunenvicary:cqm} provides a powerful graphical calculus for quantum theory. It achieves this by stripping the traditional Hilbert space model of much detail. Nevertheless, the main examples remain based on Hilbert spaces, and relations between sets. The latter can be extended to take scalars in arbitrary quantales~\cite{abramskyheunen:hstar}. This article extends scalars in the former from complex numbers to arbitrary commutative C*-algebras. In other words, we study the monoidal category of \emph{Hilbert modules} over a commutative C*-algebra. This provides a genuinely new model, that is interesting for various reasons.
\begin{itemize}
  \item Just like commutative C*-algebras are dual to locally compact Hausdorff spaces, we prove that Hilbert modules are equivalent to bundles of Hilbert spaces over locally compact Hausdorff spaces (in Section~\ref{sec:hilbertbundles}).  
  Instead of a single Hilbert space of states, we may have Hilbert spaces over every point of a base space that vary continuously. 

  \item We prove that the abstract \emph{scalars} hide more structure than previously thought: subobjects of the tensor unit correspond to clopen subsets of the base space (see Section~\ref{sec:scalars}). This exposes a rich approach to \emph{causality}~\cite{coeckelal:causal,heunenkissinger:cbh}, and opens the possibility of handling relativistic quantum information theory categorically.
  See also~\cite{enriquemolinerheunentull:space}, which additionally characterises open subsets of the base space in purely categorical terms. This also invites questions about \emph{contextuality}~\cite{abramskybrandenburger:contextuality,abramskyheunen:operational}, that might now be addressed within categorical quantum mechanics using \emph{regular logic}~\cite{heunentull:regular}.

  \item Letting the base space vary gives a bicategory of Hilbert bimodules, which forms an infinite continuous extension of the finite \emph{higher-categorical} approach to categorical quantum mechanics~\cite{vicary:higher} (see Appendix~\ref{sec:bimodules}).
\end{itemize}
We pay particular attention to \emph{Frobenius structures} (see Section~\ref{sec:frobenius}), which model classical information flow and algebras of observables~\cite{heunenvicary:cqm}. 
\begin{itemize}
  \item We prove that dagger Frobenius structures correspond to finite-dimensional C*-algebras that vary continuously over the base space (in Section~\ref{sec:cstarbundles}). The base space may be considered as modelling \emph{spacetime}. Thus spacetime protocols can still be modelled within the setting of categorical quantum mechanics~\cite{enriquemolinerheunentull:space}, and alternative models~\cite{blutecomeau:neumann} are not needed.

  In fact, we show that this correspondence of objects extends to both $*$-homomorphisms and \emph{completely positive maps} as morphisms. In other words, we identify the result of applying the CP*-construction~\cite{coeckeheunenkissinger:cpstar} to the category of Hilbert modules. 

  \item We reduce studying Frobenius structures to studying \emph{commutative} ones and \emph{central} ones (in Section~\ref{sec:transitivity}), and give nontrivial examples of each (in Section~\ref{sec:frobenius}).
  In fact, commutative Frobenius structures are equivalent to finite coverings of the base space (see Section~\ref{sec:commutativity}). 
  The proof of this fact uses that Frobenius structures have dual objects, otherwise also finite branched coverings might be allowed~\cite{pavlovtroitskii:branchedcoverings}; we leave open a characterisation of commutative H*-algebras~\cite{abramskyheunen:hstar}.
  At any rate, Frobenius structures in a category like that of Hilbert modules need not copy classical information elementwise as previously thought: there may be no copyable states at all. This more intricate structure should inform notions of classicality~\cite{heunenkissinger:cbh}. 
  On the other hand, classifying central Frobenius structures might be done using a Brauer group~\cite{auslandergoldman:brauer,raeburnwilliams:morita}, which we leave to future work.

  \item The category of Hilbert modules category captures infinite dimension, with entirely standard methods~\cite{gogiosogenovese:nonstandard}, and without dropping unitality~\cite{abramskyheunen:hstar}: although dagger Frobenius structures form local algebras of observables that are finite-dimensional, globally they can form arbitrary homogeneous C*-algebras~\cite[IV.1.6]{blackadar:operatoralgebra}. 
\end{itemize}
The article is rounded out by auxiliary results that might be expected: Hilbert modules form a symmetric monoidal dagger category with finite dagger biproducts (see Section~\ref{sec:tensor}), and the dagger dual objects are precisely the Hilbert modules that are finitely presented projective (see Section~\ref{sec:duals}).
Finally, we prove (in Section~\ref{sec:kernels}) that the category of Hilbert modules has dagger \emph{kernels} only if the base space is totally disconnected, with a view to characterising categories of Hilbert modules.
We build on results about Hilbert modules that are fragmented in the literature, but extend them to locally compact spaces, morphisms, and daggers. To keep proofs understandable, we aim for a self-contained account.

\section{Tensor products of Hilbert modules}\label{sec:tensor}

We start by recalling the basic definitions of Hilbert modules and their morphisms, which form our category of interest.
Intuitively, a Hilbert module is a Hilbert space where the base field has been replaced with a C*-algebra. In this article C*-algebras are not necessarily unital.
For more information we refer to~\cite{lance:hilbert}.

\begin{definition}
  Let $A$ be a C*-algebra. A (right) \emph{Hilbert $A$-module} is a right $A$-module $E$, equipped with a function $\inprod{-}{-}_E \colon E \times E \to A$ that is $A$-linear in the second variable, such that:
  \begin{itemize}
  \item $\inprod{x}{y}^*=\inprod{y}{x}$;
  \item $\inprod{x}{x} \geq 0$, and $\inprod{x}{x}=0$ if and only if $x=0$;
  \item $E$ is complete in the norm $\|x\|_E^2 = \| \inprod{x}{x} \|_A$.
  \end{itemize}
  A function $f \colon E \to F$ between Hilbert $A$-modules is called \emph{bounded} by $m \in \mathbb{R}$ when $\|f(x)\|_F \leq m \|x\|_E$ for all $x \in E$; in this case the infimum of such $m$ is written $\|f\|$.
  The function $f$ is called \emph{adjointable} when there exists a function $f^\dag \colon F \to E$ satisfying $\inprod{f(x)}{y}_F = \inprod{x}{f^\dag(y)}_E$ for all $x \in E$ and $y \in F$.
\end{definition}

Write $\cat{Hilb}_C\bd$ for the category of Hilbert $C$-modules and bounded $C$-linear functions.
A \emph{dagger category} is a category $\cat{C}$ with a functor $\dag \colon \C\op \to \C$ satisfying $X^\dag=X$ on objects and $f^{\dag\dag}=f$ on morphisms.
Write $\cat{Hilb}_C$ for the dagger category of Hilbert $C$-modules and adjointable functions.  

For so-called \emph{self-dual} Hilbert $A$-modules $E,F$, these two types of morphisms coincide: $\cat{Hilb}_A\bd(E,F)=\cat{Hilb}_A(E,F)$~\cite[3.3-3.4]{paschke:selfdual}.

Our next step is to show that the tensor product of Hilbert modules is well-behaved, in the sense that it makes Hilbert modules into a monoidal dagger category.

If $E$ and $F$ are Hilbert $C$-modules over a commutative C*-algebra $C$, another Hilbert $C$-module $E \otimes F$ is given by completing the algebraic tensor product $E \otimes_{\C} F$ with the following inner product and (right) $C$-module structure:
\begin{align*}
  \inprod{x_1 \otimes y_1}{x_2 \otimes y_2} & = \inprod{x_1}{x_2} \inprod{y_1}{y_2}, \\
  (x \otimes y) c & = x \otimes (yc).
\end{align*}
For more details, see Appendix~\ref{sec:bimodules}. A \emph{monoidal dagger category} is a monoidal category that is also a dagger category in which $(f \otimes g)^\dag=f^\dag \otimes g^\dag$ and the coherence isomorphisms are unitary.

\begin{proposition}\label{prop:monoidal}
  Let $C$ be a commutative C*-algebra.
  The category $\cat{Hilb}_C\bd$ is symmetric monoidal, and $\cat{Hilb}_C$ is a symmetric monoidal dagger category.
\end{proposition}
\begin{proof}
  If $f \colon E_1 \to E_2$ and $g \colon F_1 \to F_2$ are bounded maps between Hilbert $C$-modules, we may define $f \otimes g \colon E_1 \otimes F_1 \to E_2 \otimes F_2$ as the continuous linear extension of $x \otimes y \mapsto f(x) \otimes g(y)$. If $f,g$ were adjointable, then $f \otimes g$ is adjointable with adjoint $f^\dag \otimes g^\dag$:
  \begin{align*}
    \inprod{(f \otimes g)(x_1 \otimes y_1)}{x_2 \otimes y_2}
    & = \inprod{f(x_1)}{y_1} \inprod{g(y_1)}{y_2} \\
    & = \inprod{x_1}{f^\dag(y_1)} \inprod{y_1}{g^\dag(y_2)} \\
    & = \inprod{x_1 \otimes x_2}{(f^\dag \otimes g^\dag)(x_2 \otimes y_2)}.
  \end{align*}
  Clearly $\id \otimes \id = \id$ and $(f \circ g) \otimes (h \circ k) = (f \otimes h) \circ (g \otimes k)$, making the tensor product into a functor $\cat{Hilb}_C\bd \times \cat{Hilb}_C\bd \to \cat{Hilb}_C\bd$.

  There are functions $\lambda_E \colon C \otimes E \to E$, $\rho_E \colon E \otimes C \to E$, and $\alpha_{E,F,G} \colon E \otimes (F \otimes G) \to (E \otimes F) \otimes G$, that continuously extend their algebraic counterparts. Thus they satisfy the pentagon and triangle equalities. It is clear that $\alpha_{E,F,G}$ is unitary, but this is not immediate for $\lambda_E$ and $\rho_E$. Recall the precise description of the tensor product in Appendix~\ref{sec:bimodules}: it involves the $*$-homomorphism $C \to \L(E)$ that sends $f$ to $x \mapsto xf$. This $*$-homomorphism is nondegenerate~\cite[page~5]{lance:hilbert}: if $f_n$ is an approximate unit for $C$, and $x \in E$, then 
  \[
    \lim_n \inprod{x-xf_n}{x-xf_n} 
    = \lim_n \inprod{x}{x} - f_n\inprod{x}{x} - \inprod{x}{x}f_n + f_n \inprod{x}{x} f_n = 0\text{,}
  \]
  so $EC$ is dense in $E$. Now $\lambda_E \colon C \otimes E \to E$ is defined by $f \otimes x \mapsto xf$.
  Therefore
  \begin{align*}
    \| \lambda_E(\sum f_i \otimes x_i) \|_E^2
    & = \| \sum x_if_i \|_E^2 \\
    & = \| \sum \inprod{x_if_i}{x_jf_j}_E \|_C \\
    & = \| \sum \inprod{x_i}{x_j}_E f_i^* f_j \|_C\\
    & = \| \sum \inprod{f_i \otimes x_i}{f_j \otimes x_j}_{C \otimes E} \|_C\\
    & = \| \sum f_i \otimes x_i \|_{C \otimes E}^2\text{,}
  \end{align*}
  so that $\lambda_E$ is an isometric surjection $C \otimes E \to E$, and hence unitary~\cite[Theorem~3.5]{lance:hilbert}.
  Similarly, there are unitaries $\sigma_{E,F} \colon E \otimes F \to F \otimes E$ satisfying the hexagon equality.
  Thus $\cat{Hilb}_C\bd$ and $\cat{Hilb}_C$ are symmetric monoidal with unit $C$.
\end{proof}

Next, we focus on additive structure in the category of Hilbert modules. 
A \emph{zero object} is an object that is initial and terminal at the same time. If a category has a zero object, there is a unique map $0 \colon E \to F$ that factors through the zero object between any two objects.
A category has finite \emph{biproducts} when it has a zero object and any two objects $E_1, E_2$ have a product and coproduct $E_1 \oplus E_2$ with projections $p_n \colon E_1 \oplus E_2 \to E_n$ and injections $i_n \colon E_n \to E_1 \oplus E_2$ satisfying $p_n \circ i_n = \id$ and $p_m \circ i_n=0$ for $m \neq n$. A dagger category has finite \emph{dagger biproducts} when it has finite biproducts and $i_n=p_n^\dag$.

\begin{lemma}\label{lem:biproducts}
  The category $\cat{Hilb}_C\bd$ has finite biproducts; $\cat{Hilb}_C$ has finite dagger biproducts.
\end{lemma}
\begin{proof}
  Clearly the zero-dimensional Hilbert $C$-module $\{0\}$ is simultaneously an initial and terminal object.
  Binary direct sums~\cite[p5]{lance:hilbert} are well-defined Hilbert $C$-modules. Since the category $\cat{Vect}$ of vector space has finite biproducts, the universal property is satisfied via the forgetful functor $\cat{Hilb}_C\bd \to \cat{Vect}$, and it suffices to show that direct sums are well-defined on morphisms. Clearly, if $f$ and $g$ are bounded, then so is $f \oplus g$. Similarly, $f$ and $g$ are adjointable maps between Hilbert $C$-modules, so is $f \oplus g$:
  \begin{align*}
    \inprod{(f \oplus g)(x_1,y_1)}{(x_2,y_2)} 
    & = \inprod{f(x_1)}{x_2} + \inprod{g(y_1)}{y_2} \\
    & = \inprod{x_1}{f^\dag(x_2)} + \inprod{y_1}{g^\dag(y_2)} \\
    & = \inprod{(x_1,y_1)}{(f^\dag \oplus g^\dag)(x_2,y_2)}.
  \end{align*}
  Finally, the injections $E \to E \oplus F$ given by $x \mapsto (x,0)$ are clearly adjoint to the projections $E \oplus F \to E$ given by $(x,y) \mapsto x$.
\end{proof}

To conclude this preliminary section, we discuss an important aspect of the theory of Hilbert modules called localization, and show that it, too, behaves well categorically. 
Can we turn a Hilbert $C$-module into a Hilbert $D$-module? It turns out that such a change of base needs not just a map $D \to C$ to alter scalar multiplication, but also a map $C \to D$ to alter inner products.
Recall that the \emph{multiplier algebra} of a C*-algebra $A$ is the unital C*-algebra $M(A)=\cat{Hilb}_A(A,A)$, that there is an inclusion $\iota \colon A \hookrightarrow M(A)$, and that any completely positive linear map $f \colon A\to B$ extends to $M(f) \colon M(A) \to M(B)$, see~\cite[page~15]{lance:hilbert}.
A $*$-homomorphism $f \colon A \to M(B)$ is \emph{nondegenerate} when $f(A)B$ is dense in $B$.
If $A$ is already unital then $M(A)=A$.

\begin{definition}\label{def:conditionalexpectation}
  A \emph{conditional expectation} between C*-algebras $A \to B$ consists of a nondegenerate $*$-homomorphism $g \colon B \rightarrowtail M(A)$ and a completely positive linear map $f \colon A \twoheadrightarrow B$ satisfying $M(f) \circ g = \iota$. 
  A conditional expectation is \emph{strict} when $f(ab)=0$ implies $f(a)f(b)=0$ for all positive $a,b \in A$.
\end{definition}

See also Appendix~\ref{sec:radon}.

\begin{proposition}[Localization]\label{prop:localization}
  Let $f \colon C \twoheadrightarrow D$ be a conditional expectation of a unital commutative C*-algebra $C$ onto a unital commutative subalgebra $D \subseteq C$.
  There is a functor $\Loc_f \colon \cat{Hilb}_C\bd \to \cat{Hilb}_D\bd$, that sends an object $E$ to the completion of $E / N_f^E$, where $E$ is a pre-inner product $D$-module by $\inprod{x}{y}_D = f(\inprod{x}{y}_C)$, and $N_f^E = \{ x \in E \mid \inprod{x}{x}_D=0\}$.
  If $f$ is strict then it is (strong) monoidal and restricts to a dagger functor $\Loc_f \colon \cat{Hilb}_C \to \cat{Hilb}_D$. 
\end{proposition}
The functor $\Loc_f$ is called \emph{localization}~\cite[p57]{lance:hilbert}.
\begin{proof}
  On a morphism $g \colon E \to F$, the functor acts as follows.
  For $x \in E$, notice that $0 \leq |g(x)|^2 \leq \|f\|^2 |x|^2$ by~\cite[Proposition~1.2]{lance:hilbert}.
  Hence $g(N_f^E) \subseteq N_f^F$, making the function $E / N_f^E \to F / N_f^F$ given by $x + N_f^E \mapsto g(x) + N_f^F$ well-defined; define its continuous extension to be $\Loc_f(g)$. 

  This clearly respects identity morphisms and composition, making $\Loc_f$ a well-defined functor. 
  It also preserves daggers when they are available:
  \begin{align*}
    \inprod{ \Loc_f(g)(x+N_f^E)}{y + N_f^F}_{\Loc_f(F)}    
    & = f \big( \inprod{g(x)+N_f^F}{y + N_f^F}_{F} \big) \\
    & = f \big( \inprod{x+n_f^E}{g^\dag(y)+N_f^F}_{E} \big) \\
    & = \inprod{ x+n_f^E}{ \Loc_f(g^\dag)(y+N_f^F)}_{\Loc_f(E)}.
  \end{align*}

  To show that $\Loc_f$ is (strong) monoidal, we have to exhibit unitaries $D \to \Loc_f(C)$ and $\Loc_f(E) \otimes \Loc_f(F) \to \Loc_f(E \otimes F)$. For the latter, take $(x+N_f^E) \otimes (y+N_f^F) \mapsto x \otimes y + N_f^{E \otimes F}$.
  This is well-defined because $f$ is strict: if $x+N_f^E=0$, that is $f(\inprod{x}{x}_C)=0$, then $f(\inprod{x \otimes y}{x \otimes y}_C)=f(\inprod{x}{x}_C \inprod{y}{y}_C) = f(\inprod{x}{x}_C) f(\inprod{y}{y}_C) = 0$ for any $y \in F$, and so $x \otimes y \in N_f^{E \otimes F}$.
  The adjoint of this map is given by $x \otimes y + N_f^{E \otimes F} \mapsto (x + N_f^E) \otimes (y + N_f^F)$:
  \begin{align*}
    & \inprod{(x_1 + N_f^E) \otimes (y_1 + N_f^F)}{(x_2 + N_f^E) \otimes (y_2 + N_f^F)}_{\Loc_f(E)\otimes\Loc_f(F)} \\
    & = f(\inprod{x_1}{x_2}_E) \cdot f(\inprod{y_1}{y_2}_F) \\
    & = f(\inprod{x_1}{x_2}_E \cdot \inprod{y_1}{y_2}_F) \\
    & = \inprod{x_1 \otimes y_1 + N_f^{E \otimes F}}{x_2 \otimes y_2 + N_f^{E \otimes F}}_{\Loc_f(E \otimes F)}.
  \end{align*}
  This is well-defined again because $f$ is strict: if $x \otimes y \in N_f^{E \otimes F}$, that is $f(\inprod{x}{x}_C \inprod{y}{y}_C)=0$, then also $\inprod{(x+N_f^E) \otimes (y+N_f^F)}{(x+N_f^E) \otimes (y+N_f^F)}=f(\inprod{x}{x}_C) f(\inprod{y}{y}_C)=f(\inprod{x}{x}_C \inprod{y}{y}_C)=0$.
  These maps are clearly each others inverse. 

  For the unitary map $D \to \Loc_f(C)$, recall that $\Loc_f(C)$ is the completion of $C / N_f^C$ with $\inprod{c}{c'} = f(c^*c')$ and $N_f^C = \{c \in C \mid f(c^*c)=0\}$.
  Consider the map $D \to \Loc_f(C)$ given by $d \mapsto d + N_f^C$, and the map $\Loc_f(C) \to D$ given by $c+n_f^C \mapsto f(c)$. The latter is well-defined as $c-c' \in N_f^C$ implies $f(c-c')^*f(c-c')=0$ and hence $f(c)=f(c')$. 
  They are adjoint because $f$ is $D$-linear:
  \[
    \inprod{d}{f(c)}_D
    = d^* f(c)
    = f(d^* c)
    = \inprod{d}{c+N_f^C}_{\Loc_f(C)}.
  \]
  Finally, they are inverses: on the one hand $f(d)=d$ for $d \in D$; on the other hand and $c-f(c) \in N_f^C$ since
  \begin{align*}
    f( (c-f(c))^* (c-f(c))) 
    & = f(c^*c) - f(f(c)^* c) - f(c^*f(c)) + f(f(c)^* f(c)) \\
    & = f(c^*c) - f(c)^* f(c) = 0
  \end{align*}
  by the Schwartz inequality for completely positive maps~\cite[Exercise~3.4]{paulsen:positive} and~\cite[Theorem~1]{tomiyama:conditionalexpectation}.
  The required coherence diagrams are easily seen to commute. Thus $\Loc_f$ is a (strong) monoidal functor.
\end{proof}

\begin{remark}
  Not every conditional expectation is strict. For example, take $C=\C^2$, and regard $D=\C$ as a subalgebra of $C$ via $z \mapsto (z,z)$. Then $f(u,v)=u+v$ defines a conditional expectation $f\colon C \twoheadrightarrow D$. But taking $a=(1,0)$, and $b=(0,1)$ 
  shows that $f(ab)=f(0,0)=0$ but $f(a)f(b)=1 \cdot 1 = 1 \neq 0$.
  Hence for $E=F=C$, the canonical map $\Loc_f(E) \otimes \Loc_f(F) \to \Loc_f(E \otimes F)$ is not adjointable, that is, not a morphism $\cat{Hilb}_{C(X)}$. 
\end{remark}

We will be using Urysohn's lemma for locally compact spaces often~\cite[2.12]{rudin:analysis}.

\begin{lemma}[Urysohn]
  If $X$ is a locally compact Hausdorff space, and $K \subseteq V \subseteq X$ with $K$ compact and $V$ open, then there exists a continuous function $\varphi \colon X \to [0,1]$ that is 1 on $K$ and is 0 outside a compact subset of $V$.
  \qed
\end{lemma}

\begin{example}
  Any point $t$ in a locally compact Hausdorff space $X$ gives rise to a strict conditional expectation as follows. The completely positive map $f \colon C_0(X) \to \C$ evaluates at $t$.
  The $*$-homomorphism $g \colon \C \to M(C_0(X))$ is determined by $g(z)(\varphi)=z\varphi$.
  This clearly satisfies $M(f) \circ g(z)=z$, and is strict because $f$ is multiplicative.
  This \emph{localization at $t \in X$} is the setting Proposition~\ref{prop:localization} will be applied in below. 
\end{example}

\begin{remark}\label{rem:tietze}
  We will also use the previous lemma in the form of Tietze's extension theorem: 
  if $X$ is a locally compact Hausdorff space, and $K \subseteq X$ compact, then any function in $C(K)$ extends to a function in $C_0(X)$.
\end{remark}

\section{Scalars}\label{sec:scalars}

In this section, we investigate how much of the base space internalizes to the category of Hilbert modules over it. It will turn out that we need to look at morphisms into the tensor unit.

Can we get more information about $X$ from $\cat{Hilb}_{C_0(X)}$ by purely categorical means?
We first investigate \emph{scalars}: endomorphisms $I \to I$ of the tensor unit in a monoidal category. They form a commutative monoid. In the presence of biproducts, they form a semiring, and in the presence of a dagger, they pick up an involution~\cite{heunenvicary:cqm}. 

\begin{lemma}\label{lem:stonecechscalars}
  If $X$ is a locally compact Hausdorff space, there is a $*$-isomorphism between scalars of $\cat{Hilb}_{C_0(X)}$ and $C_b(X)$, the bounded continuous complex-valued functions on $X$. The same holds for $\cat{Hilb}_{C_0(X)}\bd$. 
\end{lemma}
\begin{proof}
  Recall that a closed ideal $I \subseteq A$ of a C*-algebra is \emph{essential} when $aI=\{0\}$ implies $a=0$ for all $a \in A$.
  We claim that $C_0(X)$ is an essential ideal of the C*-algebra $\L(C_0(X))$ of scalars of $\cat{Hilb}_{C_0(X)}$.
  Seeing that $C_0(X)$ is an ideal in $\L(C_0(X))$ comes down to showing that for each $f \in C_0(X)$ and scalar $s \in \L(C_0(X))$, there exists $g \in C_0(X)$ such that for all $h \in C_0(X)$ we have $hg=s(h)f$; choose $g=s(f)$.
  Seeing that the ideal is essential comes down to showing that for each scalar $s \in \L(C_0(X))$, if $s(f)g=0$ for all $f,g \in C_0(X)$, then $s=0$; given $f \in C_0(X)$, choosing $g = s(f)^*$ shows that $s(f)^*s(f)=0$ implies $\|s(f)\|^2=0$ and hence $s(f)=0$.
  It follows that the scalars of $\cat{Hilb}_{C_0(X)}$ are precisely the multiplier algebra of $C_0(X)$, which is $C_b(X)$, see~\cite[page~14--15]{lance:hilbert}.
\end{proof}

It follows that for compact $X$, the scalars in $\cat{Hilb}_{C(X)}$ simply form $C(X)$ itself: any $f \in C(X)$ gives a scalar by multiplication, and all scalars arise that way.

\begin{remark}
  If $A$ is a noncommutative C*-algebra, then $\cat{Hilb}_A$ is a perfectly well-defined dagger category. However, it cannot be monoidal with $A$ as monoidal unit. That is, Proposition~\ref{prop:monoidal} does not generalise to noncommutative $A$.
  After all, there is an injective monoid homomorphism $A \hookrightarrow \cat{Hilb}_A(A,A)$ that sends $a$ to $b \mapsto ba$, which contradicts commutativity of the latter monoid~\cite[Proposition~6.1]{kellylaplaza:compactcategories}.
\end{remark}

Next we investigate subobjects. A \emph{(dagger) subobject} of $E$ is a monomorphism $u \colon U \to E$ (satisfying $u^\dag \circ u = \id$) considered up to isomorphism of $U$.

\begin{lemma}\label{lem:daggersubobjects}
  There is an isomorphism of partially ordered sets between clopen subsets of a locally compact Hausdorff space $X$ and (dagger) subobjects of the tensor unit $C_0(X)$ in $\cat{Hilb}_{C_0(X)}$.
\end{lemma}
\begin{proof}
  We will first establish a bijection between clopen subsets of $X$ and subobjects $E \rightarrowtail C_0(X)$ such that $C_0(X) = E \oplus E^\perp$.

  Given a clopen subset $U \subseteq X$, take $E=\{f \in C_0(X) \mid f(U)=0\}$. This is a well-defined Hilbert $C_0(X)$-module under the inherited inner product $\inprod{f}{g}=f^*g$. Then $E^\perp=\{f \in C_0(X) \mid f(X \setminus U)=0\}$, and indeed $C_0(X) = E \oplus E^\perp$.

  Conversely, the image of a complemented subobject $E \rightarrowtail C_0(X)$ is a closed ideal of $C_0(X)$, and hence is of the form $E=\{f \in C_0(X) \mid f(U)=0\}$ for a closed subset $U \subseteq X$. Because the same holds for $E^\perp$ and $C_0(X)=E \oplus E^\perp$, the closed subset $U$ must in fact be clopen.
  Taking into account that subobjects are defined up to isomorphism, these two constructions are each other's inverse.

  Finally, we prove that any subobject of $C_0(X)$ in $\cat{Hilb}_{C_0(X)}$ is complemented, so that every subobject is a dagger subobject by Lemma~\ref{lem:biproducts}. See also~\cite[Theorem~3.1]{frankpaulsen:injective}.
  If $U \subseteq X$ is arbitrary, $E=\{f \in C_0(X) \mid f(U)=0\}=\{f \mid f(\overline{U})=0\}$ is a well-defined object in $\cat{Hilb}_{C_0(X)}$,
  but the inclusion $i \colon E \hookrightarrow C_0(X)$ is not necessarily a well-defined morphism. Suppose $i$ were adjointable, so that $f(t)^* g(t)=f(t)^* i^\dag(g)(t)$ for all $t \in X$ and $f,g \in C_0(X)$ with $f(U)=0$. If $t \not \in \overline{U}$,  Urysohn's lemma provides a continuous function $f \colon X \to [0,1]$ such that $f(\overline{U})=0$ and $f(t)=1$. Hence $i^\dag(g)(t)=g(t)$ for $t \in X \setminus \overline{U}$. But to make $i^\dag$ well-defined, $i^\dag(g)(t)=0$ for $t \in \overline{U}$, and $i^\dag(g)$ must be continuous. Letting $g$ range over an approximate unit for $C_0(X)$ shows that $\overline{U}$ must be clopen.
\end{proof}

It follows that there is a bijection between the clopen subsets of a locally compact Hausdorff space $X$ and self-adjoint idempotent scalars in $\cat{Hilb}_{C_0(X)}$: a dagger subobject $f \colon E \rightarrowtail C_0(X)$ induces the scalar $s=f \circ f^\dag$, and conversely, the image of a self-adjoint idempotent scalar $s \colon C_0(X) \to C_0(X)$ is a C*-subalgebra $f \colon E \rightarrowtail C_0(X)$.

\begin{lemma}\label{lem:wellpointed}
  The monoidal categories $\cat{Hilb}_{C_0(X)}$ and $\cat{Hilb}_{C_0(X)}\bd$ are \emph{monoidally well-pointed}: if $f,g \colon E_1 \otimes E_2 \to F_1 \otimes F_2$ satisfy $f \circ (x \otimes y) = g \circ (x \otimes y)$ for all morphisms $x \colon C_0(X) \to E_1$ and $y \colon C_0(X) \to E_2$, then $f=g$.
\end{lemma}
\begin{proof}
  Any element $x \in E$ gives rise to a morphism $C_0(X) \to E$ given by $\varphi \mapsto x\varphi$ with adjoint $\inprod{x}{-}_E$.
\end{proof}

\section{Hilbert Bundles}\label{sec:hilbertbundles}

Hilbert modules are principally algebraic structures. 
This section discusses a geometric description, in terms of bundles of Hilbert spaces.
While most of this material is well-known~\cite{dixmier:cstaralgebras}, we state it in a way that is useful for our purposes.
We will use the following definition of vector bundle in a Hilbert setting.

\begin{definition}\label{def:hilbertbundle}
  A \emph{Hilbert bundle} is a bundle $p \colon E \to X$ such that:
  \begin{enumerate}[label=(\alph*)]
  \item all fibres $E_t$ for $t \in X$ are Hilbert spaces;
  \item any $t_0 \in X$ has an open neighbourhood $U \subseteq X$, a natural number $n$, and sections $s_1,\ldots,s_n \colon U \to E$ such that:
  \begin{enumerate}[label=(\roman*)]
  \item $\{s_1(t),\ldots,s_n(t)\}$ is an orthonormal basis of $E_{t}$ for each $t \in U$;
  \item the map $(t,\lambda)\mapsto \sum \lambda_i s_i(t)$ is a homeomorphism $U \times \C^n \simeq E_U$.
  \end{enumerate}
  \end{enumerate}
  The \emph{dimension} of the Hilbert bundle is the function that assigns to each $t \in X$ the cardinal number $\dim(E_t)$.
  The Hilbert bundle is \emph{finite} when its dimension function is bounded: $\sup_{t \in X} \dim(E_t) < \infty$.
\end{definition}

Notice that a Hilbert bundle is a vector bundle. 
Notice also that any Hilbert bundle over a compact space $X$ is necessarily finite: because $X$ is covered by the open neighbourhoods of each $t_0 \in X$ given by (b), there is a finite subcover, and the supremum of $\dim(E_t)$ is a maximum ranging over that finite index set and is therefore always finite.

\begin{remark}\label{rem:dimensioncontinuous}
  It follows from Definition~\ref{def:hilbertbundle}(b) for a finite Hilbert bundle, the dimension $t \mapsto \dim(E_t)$ is a continuous function $X \mapsto \mathbb{N}$. 
\end{remark}

Definition~\ref{def:hilbertbundle} is a simplification of a few variations in the literature, that we now compare.
The reader that is only interested in new developments can safely skip this and continue reading at Definition~\ref{def:bundlemap}.
The \emph{$\varepsilon$-tube around a local section $s$} of a bundle $p \colon E \to X$ whose fibres are normed vector spaces is defined as
\[
  T_\varepsilon(s) = \{ x \in E \mid \forall t \in U \colon \| x - s(p(x)) \|_{E_t} < \varepsilon \}.
\]
A \emph{bounded section} $s$ is a section whose norm $\|s\| = \sup_{t \in X} \|s(t)\|$ is bounded.

\begin{definition}\label{def:fieldofhilbertspaces}
  A \emph{field of Banach (Hilbert) spaces} is a bundle $p\colon E \to X$ with:
  \begin{enumerate}
  \item all fibres $E_t$ for $t \in X$ are Banach (Hilbert) spaces;
  \item addition is a continuous function $\{(x,y) \in E^2 \mid p(x)=p(y)\} \to E$;
  \item scalar multiplication is a continuous function $\mathbb{C} \times E \to E$;
  \item the norm is a continuous function $E \to \C$;
  \item each $x_0 \in E$ has a local section $s$ with $s(p(x_0))=x_0$, 
    and $x_0$ has a neighbourhood basis $T_\varepsilon(s) \cap E_U$ for some neighbourhood $U\subseteq X$ of $p(x_0)$. 
  \end{enumerate}
  We say $p$ has \emph{locally finite rank} when:
  \begin{enumerate}[resume]
  \item any $t_0 \in X$ has a neighbourhood $U \subseteq X$ and $n \in \N$ such that $\dim(E_t)=n$ for all $t \in U$. 
  \end{enumerate}
  Finally, a field of Hilbert spaces is \emph{finite} when the dimension of its fibres is bounded.
\end{definition}

\begin{remark}\label{rem:variationsfieldofspaces}
  Definition~\ref{def:fieldofhilbertspaces} occurs in various places in the literature: 
  \begin{itemize}
    \item \cite[Definition~2.1]{dupre:classifyinghilbertbundles}: using the polarization identity we may replace (4) with inner product being a continuous function $\{(x,y) \in E^2 \mid p(x)=p(y)\}\to \C$.
    \item \cite[Definition~1]{dixmierdouady:champs} and~\cite[IV.1.6.11]{blackadar:operatoralgebra} replace (5) with the existence of a set $\Delta \subseteq \prod_{t \in X} E_t$ satisfying:
      \begin{itemize}
        \item $\{s(t) \mid s \in \Delta\}\subseteq E_t$ is dense for all $t \in X$;
        \item for every $s,s' \in \Delta$ the map $x \mapsto \inprod{s(x)}{s'(x)}_{E_t}$ is in $C(X)$;
        \item $\Delta$ is locally uniformly closed: if $s \in \prod_{t \in X} E_t$ and for each $\varepsilon>0$ and each $t \in X$, there is an $s' \in \Delta$ such that $\|s(t')-s'(t')\|<\varepsilon$ on a neighbourhood of $t$, then $s \in \Delta$;
      \end{itemize}
      this is equivalent because we can recover $E$ as $\prod_{t \in X} E_t$ with the topology generated by the basic open sets $T_\varepsilon(s) \cap E_U$ for $\varepsilon>0$, and $U \subseteq X$ open, and $s \in \Delta$; this topology makes $\Delta$ into the set of bounded sections;
    \item \cite[Definition~2.1]{dupre:classifyinghilbertbundles} explicitly takes $p$ to be open, which follows from (5), because it also considers a weaker version of (5);
    \item \cite[Definition~3.4]{takahashi:hilbertmodules} takes $s$ in (5) to be a global section, because it also considers spaces $X$ that are not functionally separated; for locally compact Hausdorff spaces $X$ this is equivalent;
    \item \emph{finite} fields of Hilbert spaces are usually called \emph{uniformly finite-dimensional}, and automatically have locally finite rank.
  \end{itemize}
  None of these variations matter for the material below.
\end{remark}

\begin{lemma}\label{lem:fieldofhilbertspaces}
  A Hilbert bundle is the same thing as a field of Hilbert spaces of locally finite rank.
  A finite Hilbert bundle is the same thing as a finite field of Hilbert spaces.
\end{lemma}
\begin{proof}
  First assume that $p \colon E \to X$ is a field of Hilbert spaces of locally finite rank.
  Condition (a) of Definition~\ref{def:hilbertbundle} is precisely condition (1) of Definition~\ref{def:fieldofhilbertspaces}.
  For condition (b), let $t_0 \in X$. Then (6) yields $n \in \N$ with $\dim(E_{t_0})=n$.
  Pick an orthonormal basis $x_1,\ldots,x_n \in E_{t_0}$. Then (5) gives continuous sections $s_1',\ldots,s_n'$ of $p$ over $U_1,\ldots,U_n\subseteq X$. Take $U=U_1 \cap \cdots \cap U_n \cap \{ t \in X \mid \{s_1(t),\ldots,s_n(t)\} \text{ linearly independent}\}$; this is an open subset of $X$ by (6) and~\cite[Proposition~1.6]{dupre:classifyinghilbertbundles}. Now, as in~\cite[Proposition~2.3]{dupre:classifyinghilbertbundles}, applying Gram-Schmidt for each $t \in U$ gives continuous sections $s_1,\ldots,s_n$ of $p$ over $U$ because of (2), (3) and (4). Moreover, these sections $s_i$ satisfy (i), (ii), and (iii) of condition (b).

  Now assume $p \colon E \to X$ is a Hilbert bundle.
  Condition (1) is still precisely condition (a).
  For condition (2), define addition $\coprod_{t_0 \in X} E_{t_0}^2 \to E_{t_0} \subseteq E$ as the cotuple of the additions $E_{t_0}^2 \to E_{t_0}$ over all $t_0 \in X$. Since the forgetful functor $\cat{Top} \to \cat{Set}$ uniquely lifts colimits, the former is continuous because the latter are continuous by (a).
  For condition (3), define scalar multiplication $\C \times E \simeq \C \times \coprod_{t_0 \in X} E_{t_0} \simeq \coprod_{t_0 \in X} \C \times E_{t_0} \to E$ as the cotuple of scalar multiplications $\C \times E_{t_0} \to E_{t_0}$ over all $t_0 \in X$. Again, this is continuous by condition (a).
  Condition (4) is satisfied exactly like (2).
  For condition (5), let $x_0 \in E$.
  Condition (b) gives a neighbourhood $U \subseteq X$ of $t_0=p(x_0)$ and $s_1,\ldots,s_n \colon U \to E$.
  Define $s \colon U \to E_U \subseteq E$  by $s(t)=\sum_i \lambda_i s_i(t)$. Then $s(p(x_0))=x_0$ by (b.ii), and $s$ is continuous on $U$. 
  Let $V\subseteq E$ be a neighbourhood of $x_0$. Find a neighbourhood $U_0 \subseteq X$ of $t_0$ with $p(V) \subseteq U_0$. Write $\varphi$ for the homeomorphism of (b.ii).
  Take $\varepsilon=1$, and $V_0 = \varphi(U_0 \times \C^n)$. Then $x_0 \in V_0 \subseteq V$ by construction, and moreover $V_0$ is contained in 
  \begin{align*}
    T_\varepsilon(s) \cap E_U
    & = \{ x \in E_U \mid \forall t \in U \colon \| x - s(p(x)) \|_{E_t} < 1 \} \\
    & = \varphi( \{ (t,\lambda_1,\ldots,\lambda_n) \in U \times \C^n \mid \| \sum_{i=1}^n \lambda_is_i(t) - s(p(\sum_{i=1}^n \lambda_i s_i(t))) \| < 1 \} )\\
    & = \varphi(U \times \C^n)
  \end{align*}
  because $s(t) \in E_t$ by (b.ii) and hence $p(s(t))=t$ by (b.i).
  Finally, condition (6) follows directly from (b).
\end{proof}

Having defined the notion of Hilbert bundle of use to us, we now define the appropriate notion of morphisms.

\begin{definition}\label{def:bundlemap}
  A \emph{bundle map} from $p \colon E \twoheadrightarrow X$ to $p' \colon E' \twoheadrightarrow X$ is a continuous function $f \colon E \to E'$ satisfying $p' \circ f = p$.
  Write $\cat{FieldHilb}_X\bd$ for the category of fields of Hilbert spaces and fibrewise linear bundle maps,
  $\cat{HilbBundle}_X\bd$ for the full subcategory of Hilbert bundles, 
  and $\cat{FHilbBundle}_X\bd$ for the full subcategory of finite Hilbert bundles.

  A bundle map $f \colon p \to p'$ between fields of Hilbert spaces is \emph{adjointable} when it is adjointable on each fibre, and the map $E'_t \ni y \mapsto f^\dag(y) \in E_t$ is continuous.
  Write $\cat{FieldHilb}_X$, $\cat{HilbBundle}_X$, and $\cat{FHilbBundle}_X$ for the wide dagger subcategories of adjointable maps.
\end{definition}

In the rest of this section we show that it is completely equivalent to work in terms of Hilbert modules, and to work in terms of Hilbert bundles. More precisely, there is a version of the Serre-Swan theorem
~\cite[13.4.5]{weggeolsen:ktheory}
for Hilbert bundles, that we now embark on proving.
We first establish a functor, then prove that it is an equivalence, and finally that they preserve monoidal structure.
If $p \colon E \twoheadrightarrow X$ is a field of Hilbert spaces, we say a function $s \colon X \to E$ \emph{vanishes at infinity} when for each $\varepsilon>0$ there is a compact $U \subseteq X$ such that $\|s(t)\|_{E_t} < \varepsilon$ for $t \in X \setminus U$.

\begin{proposition}\label{prop:sectionsvanishingatinfinity}
  Let $X$ be a locally compact Hausdorff space.
  There is a functor $\Gamma_0 \colon \cat{FieldHilb}_X\bd \to \cat{Hilb}_{C_0(X)}\bd$, defined by
  \begin{align*}
    \Gamma_0(p) & = \{ s \colon X \to E \mid p \circ s = 1_X,\, s \text{ continuous},\, s \text{ vanishes at infinity}\} \text{,}\\
    \Gamma_0(f) & = f \circ (-)\text{.}
  \end{align*}
  It restricts to a functor $\Gamma_0 \colon \cat{FieldHilb}_X \to \cat{Hilb}_{C_0(X)}$ that preserves daggers.
\end{proposition}
\begin{proof}
  Pointwise multiplication makes $\Gamma_0(p)$ into a right $C_0(X)$-module.
  For $s,s' \in \Gamma_0(p)$ and $t \in X$, the nondegenerate inner product $\inprod{s}{s'}(t) = \inprod{s(t)}{s'(t)}_{E_t}$ takes values in $C_0(X)$ by the Cauchy-Schwarz inequality.
  Finally, $\Gamma_0(p)$ is complete: if $s_n$ is a Cauchy sequence in $\Gamma_0(p)$, then $s_n(t)$ is a Cauchy sequence in $E_t$ for each $t \in X$, and hence converges to some $s(t)$; since the convergence is uniform this defines a continuous function $s \colon X \to E$, that satisfies $p \circ s = 1_X$ and vanishes at infinity by construction.
  Thus $\Gamma_0(p)$ is a well-defined Hilbert $C_0(X)$-module.

  Let $f \colon p \to p'$ be a morphism of fields of Hilbert spaces.
  Define $\Gamma_0(f) = f \circ (-) \colon \Gamma_0(p) \to \Gamma_0(p')$.
  This is clearly $C_0(X)$-linear, bounded, and functorial. It is also well-defined: if $s \in \Gamma_0(p)$, then $\|f \circ s\| \leq \|f\| \|s\|$ vanishes at infinity too.

  A morphism $f \colon p \to p'$ in $\cat{FieldHilb}_X$ is adjointable precisely when there is a bounded bundle map $f^\dag \colon p' \to p$ that provides fibrewise adjoints:
  \[
    \inprod{f(s(t))}{s'(t)}_{E_t} = \inprod{s(t)}{f^\dag(s'(t))}_{E'_t}
  \]
  for all $t \in X$, $s \in \Gamma_0(p)$, and $s' \in \Gamma_0(p')$.
  That is, $f$ is adjointable if and only if $\Gamma(f)$ is. Thus the functor $\Gamma_0$ preserves daggers.
\end{proof}

\begin{theorem}\label{thm:sectionsvanishingatinfinity}
  The functors $\Gamma_0$ from Proposition~\ref{prop:sectionsvanishingatinfinity} are equivalences.
\end{theorem}
\begin{proof}
  We first show that the functor $\Gamma_0$ is faithful. 
  Suppose $f \neq g$, say $f(x) \neq g(x)$ and $p(x)=t$.
  There exists a local continuous section $s_U \colon U \to E$ of $p$ over some open neighbourhood $U \subseteq X$ because $p$ is a field of Hilbert spaces. 
  Local compactness of $X$ ensures there is a compact neighbourhood of $t$ within $U$, which in turn contains an open neighbourhood $V \subseteq X$ of $x$.
  Urysohn's lemma provides a continuous function $r \colon X \to [0,1]$ that vanishes on $X \setminus V$ and satisfies $r(t)=1$.
  Now define $s_x \colon X \to E$ by $s_x(t)=0$ for $t \in X \setminus U$ and $s_x(t)=r(t)s_U(t)$ for $t \in U$. Then $s_x \in \Gamma_0(p)$ and $s_x(t)=x$.
  Hence $f \circ s_x(t) \neq g \circ s_x(t)$, and so $\Gamma_0(f) \neq \Gamma_0(g)$.

  Next we show that the functor $\Gamma_0$ is also full. Suppose $f \colon \Gamma_0(p) \to \Gamma_0(p')$ is bounded and $C_0(X)$-linear. For $x \in E$, set $g(x) = f(s_x)(p(x))$. 
  Because $s_x \in \Gamma_0(p)$, now $f(s_x) \in \Gamma_0(p')$, so the value $g(x)=f(s_x)(p(x))$ is an element of $E'$. Thus $g \colon E \to E'$ is a well-defined function, that furthermore satisfies $p' \circ g=p$.
  It is also fibrewise linear because if $p(x)=p(y)$ then $f(s_x+s_y)(p(x))=f(s_{x+y})(p(y))$.
  Moreover $g$ is continuous by the definition of the topology on the field of Hilbert spaces $E$. 
  Hence $g$ is a well-defined morphism of fields of Hilbert spaces.
  Finally, if $s \in \Gamma_0(p)$ and $t \in X$, then $g(s(t)) = f(s_{s(t)})(p(s(t))) = f(s_{s(t)})(t) = f(s)(t)$. So $f(s)=g \circ s$, whence $f=\Gamma_0(g)$, and $\Gamma_0$ is full.

  Finally, we show that $\Gamma_0$ is essentially surjective.
  Let $H$ be a $C_0(X)$-Hilbert module. 
  Set 
  $
    E = \coprod_{t \in X} \Loc_t(H)\text{,}
  $
  and let $p$ be the canonical projection $E \twoheadrightarrow X$.
  Because $X$ is locally compact Hausdorff, it is compactly generated: a subset $U \subseteq X$ is open if and only if $U \cap K$ is open in $K$ for all compact subsets $K \subseteq X$. Hence the topology on $X$ is determined by the topology of its compact subspaces.
  It follows from~\cite[II.1.15]{daunshofmann:sections} and~\cite[Lemma~3.01(iv), Lemma~3.09, and Proposition~3.10]{takahashi:hilbertmodules} that there is a unique weakest topology on $E$ making $p$ into a field of Hilbert spaces. 

  As in Lemma~\ref{lem:wellpointed}, we may regard elements of $H$ as adjointable maps $C_0(X) \to H$.
  For $x \in H$, define $s_x \colon X \to E$ by $s_x(t) = \Loc_t(x)$, so that $p \circ s_x = 1_X$ by construction.
  Moreover, $s_x$ vanishes at infinity, because the inner product in $H$ takes values in $C_0(X)$: if $\varepsilon>0$, there is a compact $U \subseteq X$ such that $\|s_x(t)\|_{\Loc_t(H)} = \|x\|_H(t) < \varepsilon$ for $t \in X \setminus U$.
  Finally, $s_x$ is continuous by construction of the topology on $E$. Thus $\{s_x \mid x \in X\} \subseteq \Gamma_0(p)$.

  To complete the proof that $\Gamma_0$ is essentially surjective, it now suffices to show that $\{s_x \mid x \in X\} \subseteq \Gamma_0(p)$ is dense. 
  Let $s \in \Gamma_0(p)$ and $\varepsilon$.
  Then there exists a compact subset $K \subseteq X$ such that $\|s(t)\|<\varepsilon$ for $t \in X \setminus K$.
  Urysohn's lemma provides a function $X \to [0,1]$ that vanishes at infinity such that $f(t)=1$ for $t \in K$.
  By multiplying with this function it suffices to find $x \in H$ so that the continuous local section $s_x \colon K \to X$ satisfies $\|s_x(t)-s(t)\|<\varepsilon$ for $t \in K$. This can be done by the method of the proof of~\cite[Theorem~3.12]{takahashi:hilbertmodules}.
  Therefore $\|s_x(t) - s(t)\| < \varepsilon$ for all $t \in X$.
  Thus $\Gamma_0(p) \simeq H$, and $\Gamma_0$ is essentially surjective.
\end{proof}

\begin{corollary}\label{cor:hilbertbundlemonoidal}
  The category $\cat{FieldHilb}_X\bd$ is a symmetric monoidal category for any topological space $X$, where the tensor product of $E \to X$ and $F \to X$ is $E \otimes F = \coprod_{t \in X} E_t \otimes F_t$ (with canonical topology provided by~\cite[II.1.15]{daunshofmann:sections} as in the proof of the previous lemma.) 
  The category $\cat{FieldHilb}_X$ is a symmetric monoidal dagger subcategory.
  The functors $\Gamma_0$ are (strong) monoidal.
\end{corollary}
\begin{proof}
  The tensor product $E \otimes E'$ becomes a well-defined object by letting $\Delta_{E \otimes F}$ be the closure of the pre-Hilbert $C_0(X)$-module of all finite sums of bounded sections vanishing at infinity $\sum_{i=1}^n s_i \otimes s'_i$ of $s_i \in \Gamma_0(E)$ and $s'_i \in \Gamma_0(E')$; see~\cite[Section~18]{dixmierdouady:champs} or~\cite[Definition~15.3]{bos:groupoids}.
  Via Lemma~\ref{lem:fieldofhilbertspaces}, this restricts to the monoidal product on $\cat{FHilbBundle}_X$ as in the statement.
  Defining tensor products of morphisms is straightforward, as are associators and unitors, and checking the pentagon and triangle equations. 
  The dagger is also clearly well-defined in $\cat{FHilbBundle}_X$, making it a symmetric monoidal dagger category.
  By construction of Proposition~\ref{prop:monoidal}, the functors $\Gamma_0$ are (strong) monoidal. 
\end{proof}

\section{Dual objects}\label{sec:duals}

After having given the equivalent geometric description of Hilbert modules in terms of bundles in the last section, we now return to studying the monoidal structure. This section is devoted to dual objects, that play an important role in any monoidal category. Dual objects generally behave somewhat like `finite' or `finite-dimensional' objects. The precise notion of `finiteness' in this setting turns out to be that in the following definition.

From now on we will restrict ourselves to locally compact Hausdorff spaces $X$ that are paracompact.

\begin{definition}
  A Hilbert $C$-module $E$ is \emph{finitely presented projective} when there is an adjointable map $i \colon E \to C^n$ for some $n \in \mathbb{N}$ with $i^\dag \circ i = \id[E]$. 
\end{definition}

In other words, finitely presented projective Hilbert $C$-modules are orthogonal direct summands of $C^n$. 
Any (algebraically) finitely generated projective Hilbert $C$-module is an example.
When $X$ is compact, a Hilbert $C(X)$-module is finitely presented projective if and only if it is finitely generated as a $C(X)$-module and a projective object in the category of $C(X)$-modules~\cite[Theorem~5.4.2]{weggeolsen:ktheory}.

Finitely presented projective Hilbert modules have pleasant properties, such as the following lemma, that proves that all bounded maps are adjointable in this setting.

\begin{lemma}\label{lem:finitelypresentedmapsadjointable}
  Any bounded $C$-linear map between finitely presented projective Hilbert $C$-modules is adjointable.  
\end{lemma}
\begin{proof}
  Let $i \colon E \to C^m$ and $j \colon F \to C^n$ satisfy $i^\dag \circ i =\id[E]$ and $j^\dag \circ j = \id[F]$.
  Let $f \colon E \to F$ be a bounded $C$-linear map.
  Then $g=j \circ f \circ i^\dag \colon C^m \to C^n$ is a bounded $C$-linear map, and hence an $m$-by-$n$ matrix of bounded $C$-linear maps $C \to C$. 
  But any bounded linear map $C_0(X) \to C_0(X)$ is adjointable. To see this, first use Lemma~\ref{lem:stonecechscalars} to see that it multiplies with some $k \in C_b(X)$. Now $\inprod{kl}{m}(t) = k(t)^* l(t) m(t) = \inprod{l}{k^*m}(t)$, so $k$ is adjointable.
  Thus, by Lemma~\ref{lem:biproducts}, also $g$ is adjointable.
  But then $f^\dag=i^\dag \circ g^\dag \circ j$ is an adjoint for $f$, because $\inprod{f^\dag(y)}{x}_E = \inprod{y}{j^\dag \circ g \circ i(x)}_F = \inprod{y}{j^\dag \circ j \circ f \circ i^\dag \circ i(x)}_F = \inprod{y}{f(x)}_F$.
\end{proof}

It follows that the full subcategories of $\cat{Hilb}_C$ and $\cat{Hilb}_C\bd$ of finitely presented projective Hilbert $C$-modules coincide. We write $\cat{FHilb}_C$ for this category. If $C$ is unital, we write $1_C$ for its unit.

There is an established notion of dual Hilbert module, that a priori differs from the categorical notion. The following lemma details the established notion.

\begin{lemma}\label{lem:dualmodule}
  If $X$ is a locally compact Hausdorff space, and $E$ is a finitely presented projective Hilbert $C_0(X)$-module, 
  then $E^* = \cat{Hilb}_{C_0(X)}(E,C_0(X))$ is a Hilbert $C_0(X)$-module where $\inprod{f}{g}_{E^*}$ is the element of $C_0(X)$ that $f \circ g^\dag$ multiplies with according to Lemma~\ref{lem:stonecechscalars}.
  If $X$ is compact, $\inprod{f}{g}_{E^*} = f \circ g^\dag(1_{C(X)})$.
\end{lemma}
\begin{proof}
  It is clear that $E^*$ is a $C_0(X)$-module with pointwise operations.
  Any $f,g \in E^*$ are adjointable by Lemma~\ref{lem:finitelypresentedmapsadjointable}, and hence of the form $f=\inprod{x}{-}_E$ and $g=\inprod{y}{-}_E$ for $x,y \in E$. Hence $f^\dag(\varphi)=x\varphi$ and $g^\dag(\varphi)=y\varphi$, and $f\circ g^\dag$ is the scalar that multiplies with $\inprod{x}{y}_E \in C_0(X) \subseteq C_b(X)$. Hence the inner product $\inprod{f}{g}_{E^*}=\inprod{x}{y}_E$ is well-defined. 
  It is clearly sesquilinear and positive semidefinite by Lemma~\ref{lem:stonecechscalars}.
  It is also nondegenerate: if $\inprod{f}{f}_{E^*}=0$ for $f=\inprod{x}{-}_E$, then $\inprod{x}{x}=0$, so $x=0$ and hence $f=0$.
  If $f_n$ is a Cauchy sequence in $E^*$, say $f_n=\inprod{x_n}{-}_E$, then $x_n$ is a Cauchy sequence in $E$ which converges to some $x \in E$, so $f_n$ converges to $f = \inprod{x}{-}_E$ in $E^*$.
\end{proof}

We call $E^*$ the \emph{dual} Hilbert $C$-module of $E$.

We now move from the concrete to the abstract, and define a categorical notion of dual object.

\begin{definition}\label{def:dualobjects}
  Objects $E,E^*$ in a monoidal category are called \emph{dual objects} when there are morphisms $\zeta \colon I \to E^* \otimes E$ and $\varepsilon \colon E \otimes E^* \to I$ making the following diagrams commute:
  \begin{equation}\label{eq:dualobjects}
    \begin{aligned}\begin{tikzpicture}[xscale=2]
      \node (1) at (0,1) {$E$};
      \node (2) at (.6,1) {$E \otimes I$};
      \node (3) at (2,1) {$E \otimes (E^* \otimes E)$};
      \node (4) at (2,0) {$(E \otimes E^*) \otimes E$};
      \node (5) at (.6,0) {$I \otimes E$};
      \node (6) at (0,0) {$E$};
      \draw[->] (1) to node[above]{$\rho_E^{-1}$} (2);
      \draw[->] (2) to node[above]{$\id[E] \otimes \zeta$} (3);
      \draw[->] (3) to node[left]{$\alpha_{E,E^*,E}$} (4);
      \draw[->] (4) to node[below]{$\varepsilon \otimes \id[E]$} (5);ps -
      \draw[->] (5) to node[below]{$\lambda_E$} (6);
      \draw[double distance=2pt] (6) to (1);
    \end{tikzpicture}
    \qquad
    \begin{tikzpicture}[xscale=2]
      \node (1) at (0,1) {$E^*$};
      \node (2) at (.6,1) {$I \otimes E^*$};
      \node (3) at (2,1) {$(E^* \otimes E) \otimes E^*$};
      \node (4) at (2,0) {$E^* \otimes (E \otimes E^*)$};
      \node (5) at (.6,0) {$E^* \otimes I$};
      \node (6) at (0,0) {$E^*$};
      \draw[->] (1) to node[above]{$\lambda_{E^*}^{-1}$} (2);
      \draw[->] (2) to node[above]{$\zeta \otimes \id[E^*]$} (3);
      \draw[->] (3) to node[left]{$\alpha_{E^*,E,E^*}^{-1}$} (4);
      \draw[->] (4) to node[below]{$\id[E^*] \otimes \varepsilon$} (5);ps -
      \draw[->] (5) to node[below]{$\rho_{E^*}$} (6);
      \draw[double distance=2pt] (6) to (1);
    \end{tikzpicture}\end{aligned}
  \end{equation}
  In a symmetric monoidal dagger category, dual objects are \emph{dagger dual objects} when $\zeta = \sigma \circ \varepsilon^\dag$, where $\sigma \colon E \otimes E^* \to E^* \otimes E$ is the swap map.
\end{definition}

If an object has a (dagger) dual, then that dual is unique up to unique (unitary) isomorphism. 

A priori, the two notions of dual of a Hilbert module are unrelated. We now show that the categorical notion is equivalent to the concrete notion.
In other words, we now show that dual Hilbert $C$-modules are dual objects in the finitely presented projective case over a paracompact space $X$. 

\begin{theorem}\label{thm:dualobjects}
  Let $X$ be a paracompact locally compact Hausdorff space $X$.
  For a Hilbert $C_0(X)$-module $E$, the following are equivalent:
  \begin{enumerate}
     \item[(a)] $E$ has a dagger dual object in $\cat{Hilb}_{C_0(X)}$;
     \item[(b)] $E \simeq \Gamma_0(p)$ for a finite Hilbert bundle $p$;
     \item[(c)] $E$ is finitely presented projective.
  \end{enumerate} 
\end{theorem}
Only the implication (b)$\Rightarrow$(c) requires paracompactness; (a)$\Rightarrow$(b) and (c)$\Rightarrow$(a) hold for arbitrary locally compact Hausdorff spaces $X$.
\begin{proof}
  (a)$\Rightarrow$(b): Assume that $E$ has a dagger dual object $E^*$. 
  Then also all its localisations $E_t=\Loc_t(E)$ are dagger dualisable, and so~\cite[Corollary~19]{abramskyheunen:operational} each $E_t$ is a finite-dimensional Hilbert space.
  Now regard $E$ as a field of Hilbert spaces over $X$ as in Theorem~\ref{thm:sectionsvanishingatinfinity}. Then $\zeta \colon C_0(X) \to E^* \otimes E$ and $\varepsilon \colon E \otimes E^* \to C_0(X)$ are bundle maps and hence bounded. But then $\dim(E_t) = \|\varepsilon \circ \sigma \circ \zeta(t)\| \leq \|\varepsilon\| \|\zeta\|$ is bounded, so $E$ is a finite Hilbert bundle.

  (b)$\Rightarrow$(c): 
    Let $p \colon E \twoheadrightarrow X$ be a finite Hilbert bundle.
    Then every $t \in X$ has a neighbourhood $U_t$ and a homeomorphism $g_t \colon \mathbb{C}^{n_t} \times U_t \to p^{-1}(U_t)$ that is fibrewise unitary. This forms a cover $\{U_t\}$ of $X$.
    Because $X$ is paracompact, we may pick a locally finite refinement $U_j$, and a partition of unity $f_j \colon X \to [0,1]$ subordinate to it: $f_j$ vanishes outside $U_j$ and $\sum_j f_j(t)=1$ for all $t \in X$.
    Because $p$ is finite, the numbers $n_t$ are bounded by some $n \in \mathbb{N}$, and the functions $g_t$ extend to continuous maps $\mathbb{C}^n \times U_t \to p^{-1}(U_t)$ that fibrewise satisfy $g_t \circ g_t^\dag = \id$.
    Write $g_j$ for the restrictions $\mathbb{C}^n \times U_j \to p^{-1}(U_j)$; these are still continuous and fibrewise coisometric.
    Now $(e,t)\mapsto \sum_j g_j(e,t)f_j(t)$ defines a continuous bundle map $\mathbb{C}^n \times X \twoheadrightarrow E$ that is (fibrewise) coisometric. Thus $\Gamma_0(E)$ is finitely presented projective by Theorem~\ref{thm:sectionsvanishingatinfinity}.

  (c)$\Rightarrow$(a): Assume that $i \colon E \to C_0(X)^n$ satisfies $i^\dag \circ i = \id[E]$. First, notice that $C_0(X)$ is its own dagger dual object, and therefore so is $C_0(X)^n$. Explicitly,
  \[
    \zeta \colon C_0(X) \to \big(\bigoplus\nolimits_{i=1}^n C_0(X)\big) \otimes \big(\bigoplus\nolimits_{j=1}^n C_0(X)\big) \simeq \bigoplus\nolimits_{i,j=1}^n C_0(X)
  \]
  sends $f \in C_0(X)$ to $(\delta_{i,j}f)_{i,j} \in \bigoplus_{i,j=1}^n C_0(X)$.
  Thus $(\id \otimes (i \circ i^\dag)) \circ \zeta = ((i \circ i^\dag) \otimes \id) \circ \zeta$ because it holds at each $t \in X$ and therefore globally by Theorem~\ref{thm:sectionsvanishingatinfinity}. 
  It follows that the idempotent $(\id \otimes \varepsilon) \circ (\id \otimes (i \circ i^\dag) \circ \id) \otimes (\zeta \otimes \id) \colon C_0(X)^n \to C_0(X)^n$ is split by $i \colon E \to C_0(X)^n$. 
  The unit $(i^\dag \otimes i^\dag) \circ \zeta \colon C_0(X) \to E \otimes E$ now witnesses that $E$ is a dagger dual object of $E$.
\end{proof}

Finally, we show what restricting to dual objects means concretely in terms of bundles, as described in the previous section.

\begin{theorem}\label{thm:takahashi}
  There is a monoidal equivalence of compact (dagger) categories
  \[
    \cat{FHilbBundle}_X
    \stackrel{\Gamma_0}{\longrightarrow}
    \cat{FHilb}_{C_0(X)}
  \]
  for any paracompact locally compact Hausdorff space $X$.
\end{theorem}
\begin{proof}
  By Theorem~\ref{thm:dualobjects}, the monoidal subcategory $\cat{FHilb}_{C_0(X)}$ of $\cat{Hilb}_{C_0(X)}$ is compact.
  Because (strong) monoidal functors preserve dual objects, the inverse image under $\Gamma_0$ in $\cat{FieldHilb}_X\bd$ is also compact by Corollary~\ref{cor:hilbertbundlemonoidal}.
  The dual of $E \to X$ is given by $(E^*)_t=(E_t)^*$ (with topology given by~\cite[II.1.15]{daunshofmann:sections}).
  By Lemma~\ref{lem:fieldofhilbertspaces} the functor $\Gamma_0$ therefore restricts as in the statement, and is an equivalence by Theorem~\ref{thm:sectionsvanishingatinfinity}.
\end{proof}

It follows that the tensor product of Proposition~\ref{prop:monoidal} of finitely presented projective Hilbert modules is again finitely presented projective, and so that $\cat{FHilb}_{C_0(X)}$ is a symmetric monoidal dagger category.
It is also easy to see that $\cat{FHilb}_{C_0(X)}$ has finite dagger biproducts.

\section{Frobenius structures}\label{sec:frobenius}

We now start the study of dagger Frobenius structures in the category $\cat{Hilb}_{C_0(X)}$. Many of the results below hold for arbitrary (non-dagger) Frobenius structures, but we focus on dagger Frobenius structures, and leave open the generalisation to purely algebraic proofs.
We will occassionally use the graphical calculus, in which dagger becomes horizontal reflection, tensor product becomes drawing side by side, and composition becomes vertical stacking. For more details we refer to~\cite{selinger:graphicallanguages}.
Let's start with the definitions.

\begin{definition}
  A \emph{dagger Frobenius structure} in $\cat{Hilb}_C$ is a Hilbert $C$-module $E$ with morphisms $\eta \colon C \to E$ and $\mu \colon E \otimes E \to E$ satisfying:
  \begin{align*}
    \mu \circ (\eta \otimes \id) = &\; \id = \mu \circ (\id \otimes \eta), \\
    \mu \circ (\mu \otimes \id) & = \mu \circ (\id \otimes \mu), \\
  (\mu \otimes \id) \circ (\id \otimes \mu^\dag) & = (\id \otimes \mu) \circ (\mu^\dag \otimes \id).
  \end{align*}
  or graphically:
   \begin{align*}
    \begin{pic}[yscale=-.3,xscale=-.3]
    \node[dot] (t) at (0.3,1) {};
    \node[dot] (b) at (1.3,-.5) {};
    \draw (b.east) to [out=180,in=-90] (t.north);
    \draw (b.north) to +(0,-1);
    \draw (b.west) to [out=0,in=-90] (2.3,1) to +(0,1);
   \end{pic}
   =
   \begin{pic}
    \draw (0,0) to (0,1.1);
   \end{pic}
   =
   \begin{pic}[yscale=-.3,xscale=.3]
    \node[dot] (t) at (0.3,1) {};
    \node[dot] (b) at (1.3,-.5) {};
    \draw (b.west) to [out=180,in=-90] (t.north);
    \draw (b.north) to +(0,-1);
    \draw (b.east) to [out=0,in=-90] (2.3,1) to +(0,1);
   \end{pic}
   \qquad\quad
   \begin{pic}[xscale=.3,yscale=-.3]
    \node[dot] (t) at (0,1) {};
    \node[dot] (b) at (1.3,-.5) {};
    \draw (t.west) to[out=180,in=-90] (-1,2.5);
    \draw (t.east) to[out=0,in=-90] (1,2.5);
    \draw (b.west) to[out=180,in=-90] (t.north);
    \draw (b.north) to (1.3,-1.5);
    \draw (b.east) to[out=0,in=-90] (3,2.5);
   \end{pic}
   =
   \begin{pic}[yscale=-.3,xscale=-.3]
    \node[dot] (t) at (0,1) {};
    \node[dot] (b) at (1.3,-.5) {};
    \draw (t.east) to[out=180,in=-90] (-1,2.5);
    \draw (t.west) to[out=0,in=-90] (1,2.5);
    \draw (b.east) to[out=180,in=-90] (t.north);
    \draw (b.north) to (1.3,-1.5);
    \draw (b.west) to[out=0,in=-90] (3,2.5);
   \end{pic}
   \qquad\quad
    \begin{pic}[yscale=0.5,xscale=.75]
          \node (0) at (0,0) {};
          \node (0a) at (0,1) {};
          \node [dot] (1) at (0.5,2) {};
          \node [dot] (2) at (1.5,1) {};
          \node (3) at (1.5,0) {};
          \node (4) at (2,3) {};
          \node (4a) at (2,2) {};
          \node (5) at (0.5,3) {};
          \draw (0) to (0a.center);
          \draw [out=90, in=180] (0a.center) to (1.west);
          \draw [out=0, in=180] (1.east) to (2.west);
          \draw [out=0, in=270] (2.east) to (4a.center);
          \draw (4a.center) to (4);
          \draw (2.south) to (3);
          \draw (1.north) to (5);
    \end{pic}
    =
    \begin{pic}[yscale=0.5,xscale=-.75]
          \node (0) at (0,0) {}; 
          \node (0a) at (0,1) {};
          \node [dot] (1) at (0.5,2) {};
          \node [dot] (2) at (1.5,1) {};
          \node (3) at (1.5,0) {};
          \node (4) at (2,3) {};
          \node (4a) at (2,2) {};
          \node (5) at (0.5,3) {};
          \draw (0) to (0a.center);
          \draw [out=90, in=180] (0a.center) to (1.east);
          \draw [out=0, in=180] (1.west) to (2.east);
          \draw [out=0, in=270] (2.west) to (4a.center);
          \draw (4a.center) to (4);
          \draw (2.south) to (3);
          \draw (1.north) to (5);
      \end{pic}
  \end{align*}
  A morphism $d \colon E \to E$ is called \emph{central} when $\mu \circ (\id[E] \otimes d) = d \circ \mu = \mu \circ (d \otimes \id[E])$.
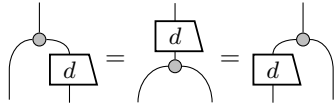
  \[
    \begin{pic}[scale=.4]
      \node[dot] (d) at (0,0) {};
      \node[morphism] (m) at (1,-1) {$d$};
      \draw (d.north) to +(0,1);
      \draw (d.east) to[out=0,in=90] (m.north);
      \draw (m.south) to +(0,-.6);
      \draw (d.west) to[out=180,in=90] (-1,-1) to (-1,-2);
    \end{pic}
    =
    \begin{pic}[scale=.4]
      \node[dot] (d) at (0,0) {};
      \node[morphism] (m) at (0,1) {$d$};
      \draw (d.north) to (m.south);
      \draw (m.north) to +(0,.6);
      \draw (d.west) to[out=180,in=90] +(-1,-1.2);
      \draw (d.east) to[out=0,in=90] +(1,-1.2);
    \end{pic}    
    =
    \begin{pic}[yscale=.4,xscale=-.4]
      \node[dot] (d) at (0,0) {};
      \node[morphism] (m) at (1,-1) {$d$};
      \draw (d.north) to +(0,1);
      \draw (d.west) to[out=0,in=90] (m.north);
      \draw (m.south) to +(0,-.6);
      \draw (d.east) to[out=180,in=90] (-1,-1) to (-1,-2);
    \end{pic}
  \]
  The dagger Frobenius structure $(E,\mu,\eta)$ is called:
  \begin{itemize}
    \item \emph{commutative} when $\mu \circ \sigma = \mu$;     
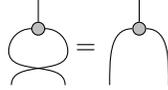
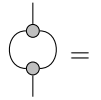
      \[
        \begin{pic}[xscale=.3,yscale=-.3]
          \node[dot] (d) {};
          \draw (d.north) to +(0,-1);
          \draw (d.west) to[out=180,in=-90] +(-1,1) to[out=90,in=-90] +(2.5,1.5);
          \draw (d.east) to[out=0,in=-90] +(1,1) to[out=90,in=-90] +(-2.5,1.5);
        \end{pic}
        =
        \begin{pic}[xscale=.3,yscale=-.3]
          \node[dot] (d) {};
          \draw (d.north) to +(0,-1);
          \draw (d.east) to[out=0,in=-90] +(1,2.5);
          \draw (d.west) to[out=180,in=-90] +(-1,2.5);
        \end{pic}
      \]
    \item \emph{special} when $\mu \circ \mu^\dag = \id[E]$; 
      \[
        \begin{pic}[scale=.5]
            \node (0) at (0,0) {};
            \node [dot] (1) at (0,1) {};
            \node [dot] (2) at (0,2) {};
            \node (3) at (0,3) {};
            \draw (0) to (1.south);
            \draw (2.north) to (3);
            \draw [in=left, out=left, looseness=1.5] (1.west) to (2.west);
            \draw [in=right, out=right, looseness=1.5] (1.east) to (2.east);
        \end{pic}
        =
        \begin{pic}
            \draw (0,0) to (0,1.25);
        \end{pic}
      \]
    \item \emph{specialisable} when $\mu \circ (d \otimes d) \circ \mu^\dag = \id[E]$ for a central isomorphism $d = d^\dag$, called the \emph{specialiser};
      
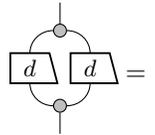
\[
        \begin{pic}[scale=.4]
          \node[dot] (b) at (0,-.25) {};
          \node[morphism] (l) at (-1,1) {$d$};
          \node[morphism] (r) at (1,1) {$d$};
          \node[dot] (t) at (0,2.25) {};
          \draw (t.north) to +(0,.7);
          \draw (b.south) to +(0,-.7);
          \draw (b.west) to[out=180,in=-90] (l.south);
          \draw (l.north) to[out=90,in=180] (t.west);
          \draw (b.east) to[out=0,in=-90] (r.south);
          \draw (r.north) to[out=90,in=0] (t.east);
        \end{pic}
        =
        \begin{pic}
          \draw (0,0) to (0,1.75);
        \end{pic}
      \]
    \item \emph{nondegenerate} when $\eta^\dag \circ \eta$ is invertible.
      \[ 
        \begin{pic}[scale=.5]
          \node[dot] (b) at (0,0) {};
          \node[dot] (t) at (0,1) {};
          \draw (b.north) to (t.south);
        \end{pic}
      \]
   \end{itemize} 
\end{definition}
It follows from the axioms of Frobenius structures that they in fact satisfy the \emph{strong Frobenius law}:
\[
  (\mu \otimes \id) \circ (\id \otimes \mu^\dag) = \mu^\dag \circ \mu = (\id \otimes \mu) \circ (\mu^\dag \otimes \id)
\]
or graphically:
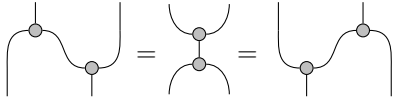
\[
    \begin{pic}[yscale=0.5,xscale=.75]
          \node (0) at (0,0) {};
          \node (0a) at (0,1) {};
          \node [dot] (1) at (0.5,2) {};
          \node [dot] (2) at (1.5,1) {};
          \node (3) at (1.5,0) {};
          \node (4) at (2,3) {};
          \node (4a) at (2,2) {};
          \node (5) at (0.5,3) {};
          \draw (0) to (0a.center);
          \draw [out=90, in=180] (0a.center) to (1.west);
          \draw [out=0, in=180] (1.east) to (2.west);
          \draw [out=0, in=270] (2.east) to (4a.center);
          \draw (4a.center) to (4);
          \draw (2.south) to (3);
          \draw (1.north) to (5);
    \end{pic}
    =
    \begin{pic}[scale=.4]
      \node[dot] (t) at (0,1) {};
      \node[dot] (b) at (0,0) {};
      \draw (b) to (t);
      \draw (t) to[out=180,in=-90] +(-1,1);
      \draw (t) to[out=0,in=-90] +(1,1);
      \draw (b) to[out=180,in=90] +(-1,-1);
      \draw (b) to[out=0,in=90] +(1,-1);
    \end{pic}
    =
    \begin{pic}[yscale=0.5,xscale=-.75]
          \node (0) at (0,0) {}; 
          \node (0a) at (0,1) {};
          \node [dot] (1) at (0.5,2) {};
          \node [dot] (2) at (1.5,1) {};
          \node (3) at (1.5,0) {};
          \node (4) at (2,3) {};
          \node (4a) at (2,2) {};
          \node (5) at (0.5,3) {};
          \draw (0) to (0a.center);
          \draw [out=90, in=180] (0a.center) to (1.east);
          \draw [out=0, in=180] (1.west) to (2.east);
          \draw [out=0, in=270] (2.west) to (4a.center);
          \draw (4a.center) to (4);
          \draw (2.south) to (3);
          \draw (1.north) to (5);
      \end{pic}
\]

Dagger Frobenius structures are their own dagger dual, with unit $\mu^\dag \circ \eta \colon I \to E \otimes E$. 
Hence dagger Frobenius structures in $\cat{Hilb}_{C_0(X)}$ live in $\cat{FHilb}_{C_0(X)}$ for paracompact $X$. 
Observe that specialisable dagger Frobenius structures are symmetric; see~\cite[Proposition~2.7]{coeckeheunenkissinger:cpstar} and notice that our notion of specialisability implies the notion of normalisability used there.

To provide some intuition we now consider the paradigmatic example of a dagger Frobenius structure.

\begin{remark}\label{rem:frobinhilbiscstar}
  For $C=\mathbb{C}$, special dagger Frobenius structures correspond to finite-dimensional C*-algebras~\cite{vicary:quantumalgebras}. 
  Any dagger Frobenius structure $E$ in $\cat{Hilb}_C$ has an involution $i \colon E \to E^*$ given by $(\id \otimes \eta^\dag) \circ (\id \otimes \mu) \circ (\zeta \otimes \id)$~\cite[4.4]{vicary:quantumalgebras}.
  In the graphical calculus, $\mu$ is drawn as $\tinymult$, and $\eta$ as $\,\tinyunit$.
  The involution is thus drawn as follows.
  
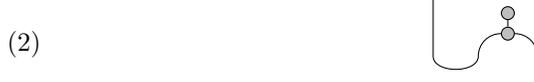
\begin{equation}\label{eq:involution}
   \begin{pic}[scale=.6]
    \node[dot] (m) at (0,0) {};
    \draw (m.north) to +(0,.3) node[dot]{};
    \draw (m.east) to[out=0,in=90] +(.5,-.5) to +(0,-.5);
    \draw (m.west) to[out=180,in=90] +(-.5,-.5) to[out=-90,in=-90] +(-1,0) to +(0,1.25);
   \end{pic}
  \end{equation}
  One of our first aims is to generalise this to arbitrary $C$. 
\end{remark}

Next we define the appropriate morphisms making dagger Frobenius structures with various properties into categories.

\begin{definition}
  A \emph{$*$-homomorphism} between Frobenius structures in $\cat{Hilb}_C$ is a morphism $f$ that preserves the involution~\eqref{eq:involution} and the multiplication: $\mu \circ (f \otimes f) = f \circ \mu$, and $f_* \circ i = i \circ f$.
  Write $\cat{Frob}_C$ for the category of specialisable dagger Frobenius structures in $\cat{Hilb}_C$ with $*$-homomorphisms, and $\cat{SFrob}_C$ for the full subcategory of special dagger Frobenius structures.
\end{definition}

The following proposition shows that, categorically, studying special Frobenius structures teaches us all about specialisable ones.

\begin{proposition}\label{prop:normal}
  The categories $\cat{Frob}_C$ and $\cat{SFrob}_C$ are monoidally equivalent (via the inclusion of the latter into the former).
\end{proposition}
\begin{proof}
  Any specialisable dagger Frobenius structure $(E,\mu,\eta)$ is isomorphic to a special one. Namely, let $d$ be the specialiser and define
  $\mu' = d \circ \mu$ and $\eta' = d^{-1} \circ \eta$.
  It is easy to check that $(E,\mu',\eta')$ is then a special dagger Frobenius structure, and that $d \colon (E,\mu',\eta') \to (E,\mu,\eta)$ and $d^{-1} \colon (E,\mu,\eta) \to (E,\mu',\eta')$ are $*$-homomorphisms that are each other's inverse. See also~\cite[Lemma~1.2]{heunenkissingerselinger:cpproj}.
\end{proof}

The following lemma observes that Frobenius structures behave well under localization, as discussed in Section~\ref{sec:tensor}.

\begin{lemma}\label{lem:localization}
  If $E$ is a special dagger Frobenius structure in $\cat{Hilb}_{C_0(X)}$, then all its localizations at $t \in X$ are dagger Frobenius structures in $\cat{Hilb}$, and hence finite-dimensional C*-algebras.
\end{lemma}
\begin{proof}
  Consider the (strong) monoidal dagger functor $\cat{Hilb}_{C_0(X)} \to \cat{Hilb}$ of Proposition~\ref{prop:localization} for each $t \in X$. Such functors preserve dagger Frobenius structures, as well as speciality and specialisability. 
\end{proof}

The following two examples look at one paradigmatic way to construct Frobenius structures in the category of Hilbert modules, generalising Remark~\ref{rem:frobinhilbiscstar}.

\begin{example}\label{ex:matrixalgebra}
  Any finite-dimensional C*-algebra $A$ is a special dagger Frobenius structure in $\cat{FHilb}$, and gives rise to a special dagger Frobenius structure $C_0(X,A)$ in $\cat{Hilb}_{C_0(X)}$ over a locally compact Hausdorff space $X$. Frobenius structures isomorphic to one of this form are called \emph{trivial}.

  In particular, $\M_n(C_0(X)) \simeq C_0(X,\M_n)$ is a special dagger Frobenius structures in $\cat{Hilb}_{C_0(X)}$. 
  It follows from Lemma~\ref{lem:biproducts} that direct sums of such matrix algebras are special dagger Frobenius structures in $\cat{Hilb}_{C_0(X)}$, too, and up to isomorphism this accounts for all trivial Frobenius structures.
\end{example}

\begin{example}\label{ex:endomorphismalgebra}
  If $X$ is a paracompact locally compact Hausdorff space, and $E$ is a finitely presented projective Hilbert $C_0(X)$-module, then $\L(E) = E^* \otimes E$ is a specialisable dagger Frobenius structure.
\end{example}
\begin{proof}
  This follows from Theorem~\ref{thm:dualobjects} and \cite[Proposition~2.11]{coeckeheunenkissinger:cpstar}; take multiplication $\id[E^*] \otimes \varepsilon \otimes \id[E] \colon E^* \otimes E \otimes E^* \otimes E \to E^* \otimes E$ and unit $\eta \colon C_0(X) \to E^* \otimes E$.
\end{proof}

Notice that trivial Frobenius structures in $\cat{Hilb}_{C_0(X)}$ in general need not be direct summands of $C_0(X)^n$. There are endomorphism algebras that are not direct sums of matrix algebras in $\cat{Hilb}_{C(X)}$. For example, take $X=2$. Then $\mathbb{M}_n(\mathbb{C})$ is a corner algebra of $\mathbb{M}_n(\C^2)$, but it is not isomorphic to a direct summand of the latter. It is nevertheless the endomorphism algebra of the Hilbert $C(X)$-module $\C^n$, but still trivial as a Frobenius structure.

The rest of this section develops nontrivial examples of commutative and central dagger Frobenius structures in $\cat{Hilb}_{C_0(X)}$. We need some topological preliminaries.

\begin{definition}\label{def:covering}
  A \emph{bundle} is a continuous surjection $p \colon Y \twoheadrightarrow X$ between topological spaces.
  Write $Y_U = p^{-1}(U)$ for $U \subseteq X$, and $Y_t=p^{-1}(t)$ for the \emph{fibre} over $t \in X$.
  The bundle is \emph{finite} when there is a natural number $n$ such that all fibres have cardinality at most $n$. 
  A \emph{(local) section} over $U$ is a continuous function $s \colon U \to Y$ satisfying $p \circ s = \id[U]$; a \emph{global section} is a section over $X$.
  A bundle is a \emph{covering} when every $t \in X$ has an open neighbourhood $U \subseteq X$ such that $Y_U$ is a union of disjoint open sets that are each mapped homeomorphically onto $U$ by $p$. 
\end{definition}

\begin{example}\label{ex:doublecouver}
  Write $S^1=\{ z \in \C \mid |z|=1 \}$ for the unit circle. For any natural number $n$, the map $p \colon S^1 \to S^1$ given by $p(z)=z^n$ is a finite covering.
  
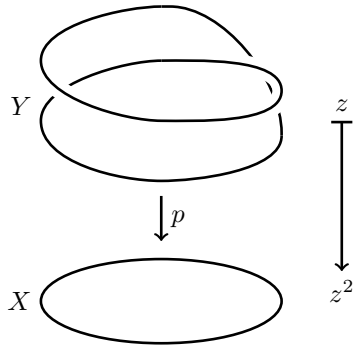
\[
    \begin{tikzpicture}[scale=.8]
      \draw[double=black,line width=3pt, double distance=1pt, draw=white, text=black] (0,2.9) to (0.15,2.9) to[out=0,in=90,looseness=.8] (2,.8) to (2,.75);
        
      \draw[double=black,line width=3pt, double distance=1pt, draw=white, text=black] (2,.75)
        to[out=-90,in=0,looseness=.7] (0,0)
        to[out=180,in=-90,looseness=.8] (-2,1)
        to[out=90,in=180,looseness=.7] (0,2);

      \draw[double=black,line width=3pt, double distance=1pt, draw=white, text=black] (0,2.9)
        to[out=180,in=90,looseness=.7] (-2,2)
        to[out=-90,in=180,looseness=.7] (0,1) 
        to[out=0,in=-90,looseness=.7] (2,1.5)
        to[out=90,in=0,looseness=.7] (0,2);

      \draw[line width=1pt] (0,-2) ellipse (2 and .7);

      \draw[line width=1pt,->] (0,-.25) to node[right]{$p$} (0,-1);
      \node[left] at (-2,-2) {$X$};
      \node[left] at (-2,1.25) {$Y$};

      \draw[line width=1pt,|->] (3,1) to (3,-1.5);
      \node[above] at (3,1) {$z$};
      \node[below] at (3,-1.5) {$z^2$};
    \end{tikzpicture}
  \]
  The map $z \mapsto z^n$ is also a finite covering on the unit disc $\{ z \in \C \mid |z| \leq 1 \}$.
\end{example}



\begin{lemma}\label{lem:diagonalclopen}
  If $p \colon Y \to X$ is a finite covering between Hausdorff spaces, then the diagonal $\Delta_Y = \{(y,y) \mid y \in Y\}$ is a closed and open subset of the pullback $Y \times_X Y =\{(y,y') \in Y \times Y \mid p(y)=p(y')\}$ of $p$ along itself.
\end{lemma}
\begin{proof}
  Because $Y$ is Hausdorff, the diagonal $\Delta_Y$ is closed in $Y \times Y$, and hence also in its closed subspace $Y \times_X Y$.
  To see that $\Delta_Y$ is also open, let $y \in Y$, say $p(y)=t$. Then the points $p^{-1}(t)=\{y_1,\ldots,y_n\}$ are distinguished by disjoint open neighbourhoods $V_1,\ldots,V_n\subseteq Y$ that are all mapped homeomorphically onto $U$ by $p$. Say $y=y_i$. Set $V = (V_i \times V_i) \cap (Y \times_X Y)$. 
  Then $(y,y) \in V$, and $V$ is open in $Y \times_X Y$ by definition of the pullback topology.
  If $v,v' \in V_i$ satisfy $p(v)=p(v')$ then $v=v'$ because $p \colon V_i \to U$ is a homeomorphism, so $V$ is contained in $\Delta_Y$.
\end{proof}

After this topological preparation, we can now construct nontrivial Frobenius structures in the category of Hilbert modules. Later, in Section~\ref{sec:commutativity}, it will turn out that any commutative Frobenius structure arises this way.

\begin{example}\label{ex:CYoverCX}
  If $p \colon Y \twoheadrightarrow X$ is a covering between locally compact Hausdorff spaces, then $C_0(Y)$ is a right $C_0(X)$-module with scalar multiplication $C_0(Y) \times C_0(X) \to C_0(Y)$ given by
  \[
    g \cdot f \colon y \mapsto g(y) \; f(p(y)).
  \]
  If $p$ is finite, then $C_0(Y)$ is a Hilbert $C_0(X)$-module under
  \[
    \inprod{f}{g}_{C_0(Y)} \colon t \mapsto \sum_{p(y)=t} \frac{f(y)^* g(y)}{|p^{-1}(t)|}.
  \]
\end{example}
\begin{proof}
  The module axioms are clearly satisfied. The inner product $\inprod{f}{g}$ is well-defined when $p$ has finite fibres; it is continuous because $p$ is a covering, and vanishes at infinity because $f$ and $g$ do so and $p$ is finite. It is clearly sesquilinear and positive definite.
  We need to prove that $C_0(Y)$ is complete in this inner product. Let $\{g_n\}$ be a Cauchy sequence in $C_0(Y)$. Say that the fibres of $p$ have cardinality at most $N$. For $\varepsilon>0$ and large $m,n$:
  \begin{align*}
    | g_n(y) - g_m(y) |^2
    & \leq \sup_{t \in X} \sum_{p(y)=t} \frac{| g_n(y) - g_m(y) |^2}{|p^{-1}(t)|^2}  \\
    & = \| \inprod{g_n-g_m}{g_n-g_m}_{C_0(Y)} \|_{C_0(X)} / N^2
    < \varepsilon 
  \end{align*}
  for all $y \in Y$, so $\{g_n(y)\}$ is a Cauchy sequence in $\C$. 
  Because this convergence is uniform, we obtain a continuous function $g \in C_0(Y)$ satisfying $g(y)=\lim g_n(y)$ pointwise, and hence also $\lim_n g_n = g$ in $C_0(Y)$. 
\end{proof}

\begin{lemma}\label{lem:coveringfrobenius}
  If $p \colon Y \twoheadrightarrow X$ is a finite covering between locally compact Hausdorff spaces, then the Hilbert $C(X)$-module $C(Y)$ of Example~\ref{ex:CYoverCX} is a nondegenerate special dagger Frobenius structure in $\cat{Hilb}_{C(X)}$.
\end{lemma}
\begin{proof}
  As multiplication $\mu \colon C_0(Y) \otimes C_0(Y) \to C_0(Y)$, take $f \otimes g \mapsto fg$. As unit $\eta \colon C_0(X) \to C_0(Y)$, take $f \mapsto f \circ p$. This clearly defines a monoid.
  Define the counit $\gamma \colon C_0(Y) \to C_0(X)$ by $\gamma(g)(t) = \sum_{p(y)=t} g(y)/|p^{-1}(t)|$. Then indeed $\eta^\dag = \gamma$:
  \[  
    \inprod{\eta(f)}{g}_{C_0(Y)}(t)
    = \sum_{p(y)=t} \frac{\overline{f(p(y))} g(y)}{|p^{-1}(t)|}
    = \overline{f(t)} \sum_{p(y)=t} \frac{g(y)}{|p^{-1}(t)|}
    = \inprod{f}{\gamma(g)}_{C_0(X)}(t).
  \]
  Then $\eta^\dag \circ \eta=\id$.
  The pushout $C_0(Y) \otimes_{C_0(X)} C_0(Y)$ of $- \circ p \colon C_0(X) \to C_0(Y)$ corresponds by Gelfand duality to $C_0(Y \times_X Y)$, where $Y \times_X Y$ is the pullback of $p$ along itself:
  \begin{align*}
    C_0(Y) \otimes_{C_0(X)} C_0(Y) & \to C_0(Y \times_X Y) \\
    f \otimes g & \mapsto \Big( (y_1,y_2) \mapsto f(y_1)g(y_2)\Big)\text.
  \end{align*}
  Define the comultiplication $\delta \colon C_0(Y) \to C_0(Y) \otimes_{C_0(X)} C_0(Y)$ by
  \[
    \delta(h)(y_1,y_2) = \left\{\begin{array}{ll}
      h(y) |p^{-1}(p(y))| & \text{ if }y_1=y_2, \\
      0 & \text{ otherwise;}
    \end{array}\right.
  \]
  this is continuous because the diagonal $\Delta_Y \subseteq Y \times_X Y$ is clopen by Lemma~\ref{lem:diagonalclopen}.
  There are $h^{(1)}_i,h^{(2)}_i \in C_0(Y)$ with $\delta = \sum_i h^{(1)}_i \otimes h^{(2)}_i$.
  Now we can verify that $\mu^\dag = \delta$; labeling $p^{-1}(t) = \{y_1,\ldots,y_n\}$:
  \begin{align*}
    \inprod{f \otimes g}{\delta(h)}_{C_0(Y) \otimes C_0(Y)}(t) 
    & = \sum_i \inprod{f}{h^{(1)}_i}_{C_0(Y)}(t) \cdot \inprod{g}{h^{(2)}_i}_{C_0(Y)}(t) \\
    & = \frac{1}{n^2} \sum_{j,l} \overline{f(y_j)g(y_l)} \sum_i h^{(1)}_i(y_j)h^{(2)}_i(y_l) \\
    & = \frac{1}{n} \sum_j \overline{f(y_j)g(y_j)}h(y_j) \\
    & = \frac{1}{n} \sum_{p(y)=t} \overline{f(y)g(y)}h(y) \\
    & = \inprod{\mu(f \otimes g)}{h}_{C_0(Y)}(t).
  \end{align*}

  Speciality $\mu \circ \mu^\dag = \id[C_0(Y)]$ is established by the following computation:
  \[
    \mu \circ \mu^\dag(h)(y)
    = \big( \sum_i h^{(1)}_i h^{(2)}_i \big)(y)
    = \sum_i h^{(1)}_i(y) h^{(2)}_i(y) 
    = h(y).
  \]
  Next, we verify the Frobenius law:
  \begin{align*}
    (\mu \otimes 1) \circ (1 \otimes \mu^\dag)(f \otimes g)
    & = (\mu \otimes 1)(f \otimes \sum_i g_i^{(1)} \otimes g_i^{(2)}) 
    = \sum_i fg_i^{(1)} \otimes g_i^{(2)}, \\
    (1 \otimes \mu) \circ (\mu^\dag \otimes 1)(f \otimes g)
    & = (1 \otimes \mu) (\sum_i f_i^{(1)} \otimes f_i^{(2)} \otimes g)
    = \sum_i f_i^{(1)} \otimes f_i^{(2)}g.
  \end{align*}
  Under the identification $C_0(Y) \otimes_{C_0(X)} C_0(Y) \simeq C_0(Y \times_X Y)$, the previous two elements of $C_0(Y) \otimes_{C_0(X)} C_0(Y)$ map $(y_1,y_2) \in Y \times_X Y$ to, respectively:
  \begin{align*}
    \sum_i (fg_i^{(1)})(y_1) g_i^{(2)}(y_2)
    & = \sum_i f(y_1)g_i^{(1)}(y_1) g_i^{(2)}(y_2) = \delta_{y_1,y_2} f(y_1)g(y_1), \\
    \sum_i (f_i^{(1)})(y_1) (f_i^{(2)}g)(y_2)
    & = \sum_i f_i^{(1)}(y_1) f_i^{(2)}(y_2)g(y_2) = \delta_{y_1,y_2} f(y_2)g(y_2).
  \end{align*}
  These are clearly equal to each other.
\end{proof}

We will see in Section~\ref{sec:commutativity} below that in fact every commutative special dagger Frobenius structure in $\cat{Hilb}_{C_0(X)}$ is of the form of the previous lemma.
Let us discuss two special cases to build intuition.

\begin{example}
  To connect to the familiar example in $\cat{FHilb}$, take $X=1$ and consider a two-point space $Y$ trivially covering $X$. Then the pullback $Y \times_X Y$ is simply the product $Y \times Y$.
  The Frobenius structure of the previous lemma then is $C_0(Y) = \C^2$.
  It carries the normalised version of its usual inner product.
  The normalisation factor is needed to make the Frobenius structure special.
  It is taken into the inner product, because otherwise the computations in the previous lemma involving the multiplication and comultiplication would become unreadable; the normalisation has to happen somewhere, and the inner product seems like the least objectionable place.
  Thus $C_0(Y)=\C^2$ is a Hilbert module over $C_0(X) = \C$.
\end{example}

\begin{example}
  Applying Lemma~\ref{lem:coveringfrobenius} to the double cover of Example~\ref{ex:doublecouver} with $n=2$, the pullback $Y \times_X Y$ is a subset $\{(a,b) \in S^1 \times S^1 \mid a^2=b^2\}$ of the torus.
  
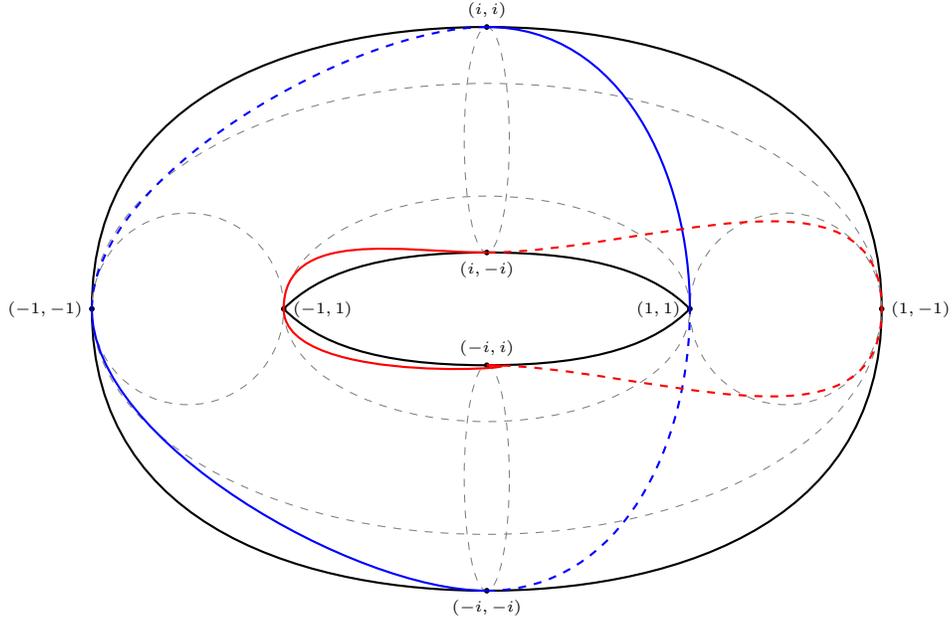
\[\begin{tikzpicture}[scale=1.5,font=\tiny]
  \draw[thick] (-3.5,0) .. controls (-3.5,2) and (-1.5,2.5) .. (0,2.5);
  \draw[thick,xscale=-1] (-3.5,0) .. controls (-3.5,2) and (-1.5,2.5) .. (0,2.5);
  \draw[thick,rotate=180] (-3.5,0) .. controls (-3.5,2) and (-1.5,2.5) .. (0,2.5);
  \draw[thick,yscale=-1] (-3.5,0) .. controls (-3.5,2) and (-1.5,2.5) .. (0,2.5);

  \draw[thick] (-1.8,0) .. controls (-1.5,-0.3) and (-1,-0.5) .. (0,-.5) .. controls (1,-0.5) and (1.5,-0.3) .. (1.8,0);

  \draw[thick] (-1.8,0) .. controls (-1.5,0.3) and (-1,0.5) .. (0,.5) .. controls (1,0.5) and (1.5,0.3) .. (1.8,0);

  \draw[dashed,gray] (-2.65,0) circle (.85); 
  \draw[dashed,gray] (2.65,0) circle (.85); 
  \draw[dashed,gray] (0,-1.5) ellipse (.2 and 1); 
  \draw[dashed,gray] (0,1.5) ellipse (.2 and 1); 

  \draw[dashed,gray] (0,0) ellipse (3.5 and 2); 
  \draw[dashed,gray] (0,0) ellipse (1.8 and 1); 

  \draw[fill=black] (0,-2.5) circle(.02) node[below]{$(-i,-i)$}; 
  \draw[fill=black] (0,-.5) circle(.02) node[above]{$(-i,i)$}; 
  \draw[fill=black] (-3.5,0) circle(.02) node[left]{$(-1,-1)$};
  \draw[fill=black] (-1.8,0) circle(.02) node[right]{$(-1,1)$};
  \draw[fill=black] (0,2.5) circle(.02) node[above]{$(i,i)$}; 
  \draw[fill=black] (0,.5) circle(.02) node[below]{$(i,-i)$}; 
  \draw[fill=black] (3.5,0) circle(.02) node[right]{$(1,-1)$};
  \draw[fill=black] (1.8,0) circle(.02) node[left]{$(1,1)$};

  \draw[thick,blue] (-3.5,0) to[out=-90,in=180,looseness=.7] (0,-2.5);
  \draw[thick,blue] (1.8,0) to[out=90,in=0] (0,2.5);
  \draw[thick,blue,dashed] (0,2.5) to[out=180,in=90,looseness=.7] (-3.5,0); 
  \draw[thick,blue,dashed] (0,-2.5) to[out=0,in=-90] (1.8,0);

  \draw[thick,red] (0,.5) to[out=180,in=90] (-1.8,0) to[out=-90,in=0] (0,-.5);
  \draw[thick,dashed,red] (0,-.5) to[out=0,in=-90] (3.5,0) to[out=90,in=0] (0,.5); 
  \end{tikzpicture}\]
  It clearly consists of two homeomorphic connected components, one of which is the diagonal $\{(a,a) \in S^1 \times S^1 \mid a \in S^1\}$, as in Lemma~\ref{lem:diagonalclopen}, and the other one is $\{(a,-a) \in S^1 \times S^1 \mid a \in S^1\}$. This enables the definition of the comultiplication $\mu^\dag$ as a map of $C(X)$-modules. 
  However, as the double cover $p$ is not trivial, it has no global sections $e_i$. Therefore there cannot be a description of the comultiplication $\mu^\dag$ in terms of $e_i \mapsto e_i \otimes e_i$ as in the case $X=1$; this is only the case over local neighbourhoods of points $t \in X$.
\end{example}

\begin{remark}
  The previous example shows that not every special dagger Frobenius structure in $\cat{Hilb}_{C_0(X)}$ is of the form $\bigoplus \mathrm{End}(E_i)$ for projective Hilbert $C_0(X)$-modules $E_i$. If that were the case, since the rank of the previous example can uniquely be written as a sum of squares as $2=1+1$, then it would have to be a direct sum of two Hilbert $C_0(X)$-modules of rank 1. But then it would have nontrivial idempotent central global sections, which it does not.
\end{remark}


We end this section with nontrivial examples noncommutative special dagger Frobenius structures in $\cat{Hilb}_C$.
In fact, we will consider examples that are noncommutative in an extreme sense, namely that of being central, defined as follows.

\begin{definition}\label{def:central}
  A dagger Frobenius structure $(E,\mu)$ in $\cat{Hilb}_C$ is \emph{central} when $Z(E)=\{x \in E \mid \forall y \in E \colon \mu(x \otimes y)=\mu(y \otimes x)\}=1_E \cdot C$ and it is faithful as a right $C$-module: $f \in C$ vanishes when $1_E f = 0$ (or equivalently, when $xf=0$ for all $x \in E$).
\end{definition}

\begin{example}
  Write $\mathbb{D} = \{ z \in \C \mid |z| \leq 1\}$ for the unit disc, 
  $S^1 = \{ z \in \mathbb{C} \mid |z| = 1\}$ for the unit circle,
  and $X=S^2 = \{ t \in \mathbb{R}^3 \mid \|t\|=1 \}$ for the 2-sphere. 
  Let $n\geq 2$ be a natural number, and consider
  \[
    E = \{ x \in C(\mathbb{D},\M_n) \mid x(z)=\diag(\overline{z},1,\ldots,1) \,x(1)\diag(z,1,\ldots,1) \text{ if }|z|=1 \}.
  \]
  Then $E$ is a $C(X)$-module via the homeomorphism $X \simeq \mathbb{D}/S^1$;
  more precisely, if $q \colon \mathbb{D} \to X$ is the quotient map, then multiplication $E \times C(X) \to E$ is given by $(x \cdot f)(z) = x(t) \cdot f(q(t))$.
  Moreover, $E$ is a Hilbert $C(X)$-module under $\inprod{x}{y}(t) = \tr(x(t)^* y(t))$.
  Finally, pointwise multiplication makes $E$ a nontrivial central special dagger Frobenius structure in $\cat{Hilb}_{C(X)}$.
\end{example}
\begin{proof}
  See~\cite[Theorem~5.8]{antonevichkrupnik:trivial} for the fact that $E$ is the Hilbert module of sections of a nontrivial finite C*-bundle. Use Theorem~\ref{thm:takahashi:algebras} below to see that it is a nontrivial special dagger Frobenius structure.

  To see that $E$ is central, notice that
  \begin{align*}
    Z(E)
    & = \{ y \in E \mid \forall x \in E \; \forall z \in \mathbb{D} \colon x(z) y(z) = y(z) x(z) \} \\
    & = E \cap C(\mathbb{D},Z(\M_n))  \\
    & = E \cap C(\mathbb{D}) = C(X) \cdot 1_E
  \end{align*}
  because if $y \in Z(E)$ does not take values in $Z(\M_n)$ at some $z \in \mathbb{D}$, there are two cases: if $|z|<1$ or $z=1$, then $x$ does not commute with some $y \in E$ at $z$; and if $|z|=1$, then it also does not take values in $Z(\M_n)$ at $z=1$.
\end{proof}

\section{C*-bundles}\label{sec:cstarbundles}

We know from Section~\ref{sec:hilbertbundles} that Hilbert modules are equivalent to Hilbert bundles. We are interested in Frobenus structures in the category of Hilbert modules, as defined in Section~\ref{sec:frobenius}. In this section, we apply the bundle perspective to dagger Frobenius structures. They form C*-algebras themselves, as the following lemma shows.

\begin{lemma}\label{lem:frobeniuscstar}
  Special dagger Frobenius structures in $\cat{Hilb}_{C_0(X)}$ are C*-algebras.
\end{lemma}
\begin{proof}
  First of all, $E$ is clearly a Banach space, as an object in $\cat{Hilb}_{C_0(X)}$. 
  It is also an algebra with multiplication $\mu \colon E \otimes E \to E$.
  In fact, it becomes a Banach algebra because $\mu^\dag \mu$ is a projection by speciality~\cite[Lemma~9]{abramskyheunen:operational}:
  \begin{align*}
    \|xy\|^2
    & = \| \inprod{\mu(x \otimes y)}{\mu(x \otimes y)}_{E\otimes E} \|_{C_0(X)} \\
    & = \| \inprod{\mu^\dag \mu(x \otimes y)}{x \otimes y}_{E\otimes E} \|_{C_0(X)} \\
    & \leq \| \inprod{x \otimes y}{x \otimes y}_{E\otimes E} \|_{C_0(X)} \\
    & = \| \inprod{x}{x}_E \inprod{y}{y}_E \|_{C_0(X)} \\
    & \leq \| \inprod{x}{x}_E \|_{C_0(X)} \| \inprod{y}{y}_E \|_{C_0(X)} \\
    & = \|x\|^2 \|y\|^2.
  \end{align*}
  Finally, this satisfies the C*-identity because it does so locally at each $t \in X$ by Lemma~\ref{lem:localization}:
  \[
    \| x^* x \|_E
    = \sup_{t \in X} \| x^* x \|_{\Loc_t(E)}
    = \sup_{t \in X} \| x \|_{\Loc_t(E)}^2
    = \| x \|_E^2.
  \]
  The outer equalities use Theorem~\ref{thm:takahashi}.
\end{proof}

The C*-algebras induced by dagger Frobenius structures have more internal structure: they are in fact a bundle of C*-algebras, as made precise in the following definition.

\begin{definition}\label{def:cstarbundle}
  A \emph{finite (commutative) C*-bundle} is a bundle $p \colon E \twoheadrightarrow X$ where:
  \begin{enumerate}[label=(\arabic*)]
  \item all fibres $E_t$ for $t \in X$ are finite-dimensional (commutative) C*-algebras;
  \item any $t_0 \in X$ has an open neighbourhood $U \subseteq X$, a finite-dimensional C*-algebra $A$, and a homeomorphism $\varphi \colon U \times A \to E_U$, such that the map $\varphi(t,-) \colon A \to E_t$ is a $*$-isomorphism for each $t \in U$;
  \item the dimension of the fibres is bounded.
  \end{enumerate}
\end{definition}

If $X$ is compact, then condition (3) is superfluous.

The next lemma shows that we may view being a finite C*-bundle as structure laid on top of being a finite Hilbert bundle.

\begin{lemma}\label{lem:cstarbundleishilbertbundle}
  Any finite C*-bundle is a finite Hilbert bundle.
\end{lemma}
\begin{proof}
  Let $p \colon E \to X$ be a finite C*-bundle. 
  The fibre over $t_0 \in X$ is a finite-dimensional C*-algebra, and hence canonically of the form $\M_{n_1}\oplus\cdots\oplus\M_{n_k}$ up to isomorphism.
   It is a finite-dimensional Hilbert space under the inner product
  \[
    \inprod{(a_1,\ldots,a_k)}{(b_1,\ldots,b_k)} = \tr(a_1^*b_1) + \cdots + \tr(a_k^*b_k).
  \]
  Condition (2) also gives an open neighbourhood $U$ of $t_0$, a finite-dimensional C*-algebra $A=\M_{n_1}\oplus\cdots\oplus\M_{n_k}$, and a homeomorphism $\varphi \colon U \times A \to E_U$. Take $n=\dim(A)$, and let the standard matrix units constitute an orthonormal basis $e_1,\ldots,e_n$ of $A$. Define  continuous sections $s_i \colon U \to E$ by $s_i(t) = \varphi(t,e_i)$. Now $\{s_i(t)\}$ forms an orthonormal basis of $E_t$ for all $t \in U$ by (2). 
\end{proof}

\begin{example}\label{ex:pantsbundle}
  If $X$ is a paracompact locally compact Hausdorff space, and $E$ a finitely presented projective Hilbert $C_0(X)$-module, then $\L(E) = E^* \otimes E \simeq \cat{Hilb}_{C_0(X)}(E,E)$ is a finite C*-bundle.
\end{example}
\begin{proof}
  Notice that $\cat{Hilb}_{C_0(X)}$ is a C*-category~\cite[Example~1.4]{ghezlimaroberts:wstarcategories}, and a monoidal category by Proposition~\ref{prop:monoidal}. Thus it is a tensor C*-category, and hence a 2-C*-category (with a single object).
  The result follows from~\cite[Proposition~2.7]{zito:cstarcategories}.
\end{proof}

As in Section~\ref{sec:hilbertbundles}, let us spend some time on connecting to terminology in the literature. The reader only interested in new developments may safely skip may safely skip the next lemma.

Just as Definition~\ref{def:hilbertbundle} was a simplification of Definition~\ref{def:fieldofhilbertspaces}, the previous definition is a simplification of the notion of \emph{field of C*-algebras} in the literature~\cite{fell:bundles,felldoran:bundles,dupre:classifyinghilbertbundles,tomiyama:representation,tomiyamatakesaki:bundles,blackadar:operatoralgebra}: a field $p \colon E \twoheadrightarrow X$ of Banach spaces where each fibre is a C*-algebra, where multiplication gives a continuous function $\{(x,y) \in E^2 \mid p(x)=p(y)\} \to E$, and where involution gives a continuous function $E \to E$. A field of C*-algebras is \emph{uniformly finite-dimensional} when each fibre is finite-dimensional, and the supremum of the dimensions of the fibres is finite.

\begin{lemma}
  A finite C*-bundle is the same thing as a uniformly finite-dimensional field of C*-algebras.
\end{lemma}
\begin{proof}
  By Lemma~\ref{lem:cstarbundleishilbertbundle}, any finite C*-bundle is a finite Hilbert bundle, and hence a finite field of Banach spaces of locally finite rank by Lemma~\ref{lem:fieldofhilbertspaces}. Similarly, multiplication and involution are continuous functions by the same argument as in the proof of Lemma~\ref{lem:fieldofhilbertspaces}. 

  The converse is similar to Lemma~\ref{lem:fieldofhilbertspaces} for the most part.
  Let $p \colon E \twoheadrightarrow X$ be a uniformly finite-dimensional field of C*-algebras.
  Let $t_0 \in X$. Take $A=E_{t_0}$, say of the form $\M_{n_1}\oplus\cdots\oplus\M_{n_k}$, and let $x_1,\ldots,x_n$ be the orthonormal basis of $A$ constituted by standard matrix units.
  Condition (5) gives sections $s_i \colon U \to X$ with $s_i(t_0)=x_i$.
  Take $U=U_1 \cap \cdots \cap U_n \cap \{ t \in X \mid \{s_i(t)\} \text{ linearly independent}\}$; this is an open subset of $X$.
  Define $\varphi \colon U \times A \to E_U$ by linearly extending $(t,s_i)\mapsto s_i(t)$. This is a homeomorphism, and $\varphi(t,-)$ is a $*$-isomorphism by construction.
\end{proof}

Next, we turn to the appropriate notion of morphism between finite C*-bundles.

\begin{definition}
  A \emph{morphism} of finite C*-bundles is a bundle map that is fibrewise a $*$-homomorphism. 
  Write $\cat{FCstarBundle}_X$ for the category of finite C*-bundles with their morphisms.
\end{definition}

We are now ready for the main result of this sectoin: to characterise the (commutative) dagger specialisable Frobenius structures in $\cat{Hilb}_{C_0(X)}$ as finite (commutative) C*-bundles over $X$.

\begin{theorem}\label{thm:takahashi:algebras}
  There is an equivalence of monoidal dagger categories
  \[\begin{tikzpicture}
    \node (M) at (4,0) {$\cat{Frob}_{C_0(X)}$};
    \node (F) at (0,0) {$\cat{FCstarBundle}_X$};
    \draw[->] (F.east) to node[above] {$\Gamma_0$} (M.west);
  \end{tikzpicture}\]
  for any paracompact locally compact Hausdorff space $X$.
\end{theorem}
\begin{proof}
  By Proposition~\ref{prop:normal}, we may use $\cat{SFrob}_{C_0(X)}$ instead of $\cat{Frob}_{C_0(X)}$.
  Write $\Delta$ for the adjoint of $\Gamma_0$ of Theorem~\ref{thm:takahashi}.
  Let $(E,\mu,\eta)$ be a special dagger Frobenius structure in $\cat{Hilb}_{C_0(X)}$. 
  Equivalently, the embedding $R \colon E \to \L(E)$
  \[\begin{pic}[scale=.5]
    \node[dot] (d) {};
    \draw (d.north) to +(0,.5) node[right] {$E$};
    \draw (d.east) to[out=0,in=90] +(.5,-.5) to +(0,-.5) node[right] {$E$};
    \draw (d.west) to[out=180,in=90] +(-.5,-.5) to[out=-90,in=-90] +(-1,0) to +(0,1.25) node[left] {$E^*$};
  \end{pic}\]
  and the involution $i \colon E \to E^*$ of equation~\eqref{eq:involution} satisfy $i \circ R = R_* \circ i$~\cite[Corollary~9.7]{heunenkarvonen:monads}. 
  By Example~\ref{ex:pantsbundle}, $\Delta(E^* \otimes E)$ is a finite C*-bundle over $X$.
  Now, because both $i$ and $R$ are defined purely in terms of tensor products,  composition, and dagger, the above equations also hold fibrewise by Theorem~\ref{thm:takahashi}. 
  Hence $\Delta(E)$ is a finite Hilbert bundle, which embeds into $\Delta(E^* \otimes E)$ with $\Delta(R)$, and is closed under the involution $\Delta(i)$. 
  We conclude that $\Delta(E)$ is in fact a finite C*-bundle.
  The same reasoning establishes the converse: if $p$ is a finite C*-bundle, then $\Gamma_0(p)$ is a special(isable) dagger Frobenius structure in $\cat{Hilb}_{C(X)}$.
  Compare~\cite[Definition~21.7]{dixmierdouady:champs}.
  See also~\cite{yamagami:duality}.
\end{proof}

The rest of this section derives from the previous theorem some corollaries of interest to categorical quantum mechanics. We start with the phase group.

Recall that the \emph{phase group} of a dagger Frobenius structure $E$ consists of all morphisms $\phi \colon C_0(X) \to E$ satisfying
$(\phi^\dag \otimes \id) \circ \mu^\dag \circ \phi = \eta = (\id \otimes \phi^\dag) \circ \mu^\dag \circ \phi$~\cite{heunenvicary:cqm}. 
A \emph{group bundle} is a bundle $E \twoheadrightarrow X$ whose every fibre is a group, and such that each point $t_0 \in X$ has a group $G$ and a neighbourhood on which fibres are isomorphic to $G$.
Recall that the unitary group of a unital C*-algebra $E_t$ is $\{u \in E_t \mid uu^*=u^*u=1\}$.

\begin{corollary}\label{cor:phasegroup}
  The phase group of a dagger Frobenius structure $E$ in $\cat{FHilb}_{C_0(X)}$ is a group bundle $U(E) \twoheadrightarrow X$ whose fibres are the unitary groups of fibres of $E$. 
\end{corollary}
\begin{proof}
  The general case follows easily from the case $X=1$, which is a simple computation~\cite{heunenvicary:cqm}. 
\end{proof}

For example, for the trivial Frobenius structure $C_0(X)$ in $\cat{FHilb}_{C_0(X)}$, the phase group is the trivial bundle $U(1) \times X \twoheadrightarrow X$. 

We end this section by considering a more permissive notion of morphism between Frobenius structures, namely completely positive maps.

\begin{definition}
  A \emph{completely positive map} between finite C*-bundles over $X$ is a bundle map that is completely positive on each fibre. Write $\cat{FCstarBundle}_X^\cp$ for the category of finite C*-bundles and completely positive maps.
\end{definition}

In general, there is a construction that takes a monoidal dagger category $\cat{C}$ to a new one $\CPs[\cat{C}]$, see~\cite{coeckeheunenkissinger:cpstar}. Objects in $\CPs[\cat{C}]$ are special dagger Frobenius structures in $\cat{C}$. Morphisms $(E,\tinymult[whitedot]) \to (F,\tinymult)$ in $\CPs[\cat{C}]$ are morphisms $f \colon E \to F$ in $\cat{C}$ with
\begin{equation}\label{eq:cpstarcondition}
    \begin{pic}[yscale=.66]
      \node[morphism] (f) at (0,0) {$f$};
      \node[whitedot] (l) at (-.5,-1) {};
      \node[dot] (r) at (.5,1) {};
      \draw (l.east) to[out=0,in=-90] (f.south);
      \draw (f.north) to[out=90,in=180] (r.west);
      \draw (r.north) to (.5,2) node[above] {$F$};
      \draw (l.south) to (-.5,-2) node[below] {$E$};
      \draw (r.east) to[out=0,in=90,looseness=.4] (1,-2) node[below] {$F$};
      \draw (l.west) to[out=180,in=-90,looseness=.4] (-1,2) node[above] {$E$};
    \end{pic}
    \qquad = \qquad
    \begin{pic}
      \node[morphism,hflip] (t) at (0,.55) {$g$};
      \node[morphism] (b) at (0,-.55) {$g$};
      \draw (b.north) to node[right] {$G$} (t.south);
      \draw ([xshift=-1mm]t.north west) to +(0,.5) node[above] {$E$};
      \draw ([xshift=1mm]t.north east) to +(0,.5) node[above] {$F$};
      \draw ([xshift=-1mm]b.south west) to +(0,-.5) node[below] {$E$};
      \draw ([xshift=1mm]b.south east) to +(0,-.5) node[below] {$F$};
    \end{pic}
\end{equation}
for some object $G$ and some morphism $g \colon E \otimes F \to G$ in $\cat{C}$.

\begin{theorem}\label{thm:cpstar}
  There is an equivalence of compact dagger categories
  \[\begin{tikzpicture}
    \node (M) at (4,0) {$\CPs(\cat{Hilb}_{C_0(X)})$};
    \node (F) at (0,0) {$\cat{FCstarBundle}_X^\cp$};
    \draw[->] (F.east) to node[above] {$\Gamma_0$} (M.west);
  \end{tikzpicture}\]
  for any paracompact locally compact Hausdorff space $X$.
\end{theorem}
\begin{proof}
  The correspondence on objects is already clear from Theorem~\ref{thm:takahashi:algebras}.
  By definition, morphisms in $\CPs(\cat{Hilb}_{C_0(X)})$ are morphisms in $\cat{Hilb}_{C_0(X)}$ that satisfy~\eqref{eq:cpstarcondition}. Because the equivalence is monoidal, these correspond to morphisms between finite C*-bundles that satisfy the same condition. By Theorem~\ref{thm:takahashi:algebras} the condition also holds in each fibre. Hence~\cite{coeckeheunenkissinger:cpstar} these morphisms are completely positive maps in each fibre.
\end{proof}

\section{Commutativity}\label{sec:commutativity}

In this section we will completely characterise the commutative special dagger Frobenius structures in the category of Hilbert modules.
By Theorem~\ref{thm:takahashi:algebras}, they correspond to commutative finite C*-bundles. In this section we phrase that in terms of Gelfand duality, generalizing~\cite{pavlovtroitskii:branchedcoverings}. 
We first reduce to nondegenerate Frobenius structures. 

\begin{lemma}\label{lem:nonzerodimension}
  Let $X$ be a locally compact Hausdorff space.
  Any (specializable) dagger Frobenius structure in $\cat{Hilb}_{C_0(X)}$ is determined by
  a nondegenerate (specializable) one in $\cat{Hilb}_{C_0(U)}$ for a clopen subset $U \subseteq X$.
\end{lemma}
\begin{proof}
  Let $E \in \cat{Frob}_{C_0(X)}$. By Theorem~\ref{thm:takahashi:algebras} it corresponds to a finite C*-bundle. (Note that this does not need paracompactness.) So $t \mapsto \dim(E_t)$ is a continuous function $X \to \mathbb{N}$, and $U=\{t \in X \mid \dim(E_t)>0\}$ is clopen.
  We need to show that the restricted finite C*-bundle over $U$ is nondegenerate. Note that $\dim(E_t)$ is the value of the scalar $\eta^\dag \circ \mu \circ \mu^\dag \circ \eta \in C_b(X)$ at $t$. In particular, it takes values in $\mathbb{N}$, and if $t \in U$, then it is invertible. 
\end{proof}

Next, we show that any nondegenerate specialisable dagger Frobenius structure in $\cat{Hilb}_{C_0(X)}$ is induced by a finite bundle $p \colon Y \twoheadrightarrow X$. 

\begin{proposition}\label{prop:finitefibres}
  Let $X$ be a paracompact locally compact Hausdorff space.
  Any commutative nondegenerate specialisable dagger Frobenius structure in $\cat{Hilb}_{C_0(X)}$ is isomorphic as a $*$-algebra to $C_0(Y)$ for some locally compact Hausdorff space $Y$ through a finite bundle $p \colon Y \twoheadrightarrow X$.
\end{proposition}
\begin{proof}
  By Proposition~\ref{prop:normal} 
   we may assume that the given dagger Frobenius structure $E$ is special.
  It then follows from Lemma~\ref{lem:frobeniuscstar} that $E$ is of the form $C_0(Y)$ for some locally compact Hausdorff space $Y$. 
  Applying Lemmas~\ref{lem:wellpointed} and~\ref{lem:stonecechscalars} to the unit law $\mu \circ (\eta \otimes \eta) = \eta \circ \lambda$ shows that the map $\eta \colon C_0(X) \to C_0(Y)$ is multiplicative. Being a morphism in $\cat{Hilb}_{C_0(X)}$ it is also additive. It preserves the involution by definition of dual objects.
  Hence $\eta$ is a $*$-homomorphism, which is nondegenerate as in Proposition~\ref{prop:monoidal}.
  By Gelfand duality, therefore $\eta$ is of the form $- \circ p \colon C_0(X) \to C_0(Y)$ for a continuous map $p \colon Y \to X$.
  Because $\eta \colon C_0(X) \to C_0(Y)$ is injective by nondegeneracy, $p$ is surjective.

  The complex vector space $C_0(p^{-1}(t))$ contains at least as many linearly independent elements as distinct elements $y_i$ of
  $p^{-1}(t)$, 
  namely the continuous extension of $y_j \mapsto \delta_{ij}$ by Tietze's extension theorem.
  But $C_0(Y)$ is finitely presented projective as a $C_0(X)$-module by Theorem~\ref{thm:dualobjects}, so
  there is a natural number $n$ and some $E \in \cat{FHilb}_{C_0(X)}$ such that for each $t \in X$ we have $C_0(p^{-1}(t)) \oplus E_t \simeq \mathbb{C}^n$ by localising as in Proposition~\ref{prop:localization}. 
  Thus $\dim(C_0(p^{-1}(t))) \leq n$, and hence $p^{-1}(t)$ has cardinality at most $n$, for each $t \in X$.
\end{proof}

Our next goal is to show that the finite bundle $p$ is of the form of Lemma~\ref{lem:coveringfrobenius}. We will do this in several steps.
To show that $p$ must in fact be a finite covering, we first prove $p$ is an open map.



\begin{lemma}\label{lem:open}
  Let $X$ be a paracompact locally compact Hausdorff space.
  Nondegenerate commutative specialisable dagger Frobenius structures in $\cat{Hilb}_{C_0(X)}$ are of the form $C_0(Y)$ for a finite bundle $p \colon Y \twoheadrightarrow X$ that is open.
\end{lemma}
\begin{proof}
  By Theorem~\ref{thm:takahashi:algebras} a specialisable dagger Frobenius structure $E$ in $\cat{Hilb}_{C_0(X)}$ corresponds to a finite C*-bundle, whose fibres have uniformly bounded dimension.
  We need to show that $p$ is open; suppose for a contradiction that it is not.
  Let $V \subseteq Y$ be an open set such that $p(V) \subseteq X$ is not open.
  Fix a limit point $t_0 \in p(V)$ of $X\setminus p(V)$, and pick $s_0 \in V$ with $p(s_0)=t_0$. Urysohn's lemma now provides a continuous function $y \colon Y \to [0,1]$ with $y(s_0)=1$ that vanishes outside a compact subset of $V$ and hence vanishes at infinity.
  Now $\eta^\dag(y)(t)=0$ if and only if $\sum_{p(s)=t} y(s)=0$ for all $t \in X$, so $\eta^\dag(y)$ vanishes on $X \setminus p(V)$.
  But $\eta^\dag(y)(t_0)>0$ by Lemma~\ref{lem:nonzerodimension}, contradicting continuity of $\eta^\dag$.
  See also~\cite[Theorem~5.6]{pavlovtroitskii:branchedcoverings}, \cite[2.2.3]{blanchardkirchberg:glimmhalving} and~\cite[Theorem~4.3]{ivankov:quantization}.
\end{proof}

Next, we show that $p \colon Y \twoheadrightarrow X$ must also be a closed map. When $Y$ is compact and $X$ is Hausdorff this is automatic because continuous images of compact spaces are compact and compact subsets of Hausdorff spaces are closed; we show that it also holds when $Y$ is only locally compact.

\begin{lemma}\label{lem:closed}
  Finite bundles $p \colon Y \twoheadrightarrow X$ of locally compact Hausdorff spaces are closed.
\end{lemma}
\begin{proof}
  Suppose $V\subseteq Y$ is closed. We want to show that $U=p(V) \subseteq X$ is closed. Let $t_\alpha$ be a net in $U$ that converges to $t \in X$. Pick $s_\alpha$ in $p^{-1}(t_\alpha) \cap V$. Say $p^{-1}(t) \cap V = \{s_1,\ldots,s_n\}$. Pick compact neighbourhoods $V_i \subseteq V$ of $s_i$ (possible because $Y$ is locally compact). Then $s_\alpha$ is eventually in $\bigcup_i V_i$ (because this finite union is compact). So a subnet of $s_\alpha$ converges to one of the $s_i \in V$. But then, by continuity of $p$, a subnet of $t_\alpha$ converges to $p(s_i) \in U$. But then $t=p(s_i)$ is in $U$ (because $X$ is Hausdorff).    
\end{proof}



Finally, we can show that the $p$ must be a finite covering.

\begin{proposition}\label{prop:covering}
  Let $X$ be a paracompact locally compact Hausdorff space. 
  Any nondegenerate commutative specialisable dagger Frobenius structure in $\cat{Hilb}_{C_0(X)}$ is of the form $C_0(Y)$ for a finite covering $p \colon Y \twoheadrightarrow X$.
\end{proposition}
\begin{proof}
  We simplify~\cite[Theorem~4.4]{pavlovtroitskii:branchedcoverings}.
  By Theorem~\ref{thm:dualobjects}, $C_0(Y) \oplus E \simeq C_0(X)^n$ for some $n \in \mathbb{N}$ and $E \in \cat{FHilb}_{C_0(X)}$. 
  Hence $k_t = |\dim(C_0(Y)_t)|=|p^{-1}(t)| \leq n$ for all $t \in X$. 
  Because $t \mapsto k_t$ is a continuous function $X \to \mathbb{N}$ by Remark~\ref{rem:dimensioncontinuous}, 
  the subsets $X_k = \{ t \in X \mid k_t = k \}\subseteq X$ are closed and open for $k=1,\ldots,n$.
  That is, $X=X_1 \sqcup \cdots \sqcup X_n$ is a finite disjoint union of clopen subsets, on each of which the fibres of $p$ have the same cardinality.

  Now for $t \in X$, by Lemma~\ref{lem:closed} and~\cite[Lemma~2.2]{pavlovtroitskii:branchedcoverings}, 
  we can choose a neighbourhood $U \subseteq X$ over which $p^{-1}(U)$ is a disjoint union of open subsets $V_1,\ldots,V_k \subseteq Y$ that each contain a preimage of $t$. 
  By replacing $U$ by $\bigcap p(V_i)$, and intersecting $V_i$ with $\bigcap p^{-1}(p(V_j))$, we may assume that each $p \colon V_i \to U$ is surjective. But then, because all fibres have the same size, it cannot happen that one of the $V_i$ has two points of a fibre, as then another $V_j$ must have none (because there are only finitely many points in the fibre), whence $p \colon V_j \to U$ would not be surjective. So each $p \colon V_i \to U$ is a closed and open bijection, and hence a homeomorphism.
\end{proof}

This completely characterises commutative specialisable dagger Frobenius structures in $\cat{Hilb}_{C_0(X)}$ for paracompact connected $X$.
Write $\cat{cFrob}_{C_0(X)}$ for the full subcategory of nondegenerate commutative objects in $\cat{Frob}_{C_0(X)}$, and write $\cat{Covering}_X$ for the category of finite coverings and bundle maps.
The category $\cat{Covering}_X$ is symmetric monoidal under Cartesian product.

\begin{theorem}\label{thm:commutativecase}
  For any paracompact locally compact Hausdorff space $X$ there is an equivalence $\cat{cFrob}_{C_0(X)} \simeq \cat{Covering}_X$ of symmetric monoidal dagger categories.
\end{theorem}
\begin{proof}
  Combine Lemma~\ref{lem:open} and Lemma~\ref{lem:coveringfrobenius} to establish the equivalence. Monoidality follows because the tensor product is the coproduct of commutative C*-algebras, and so $C_0(X) \otimes C_0(Y) \simeq C_0(X) + C_0(Y) \simeq C_0(X \times Y)$ by duality.
\end{proof}

Alternatively, we could include degenerate objects in $\cat{cFrob}_{C_0(X)}$ and objects $p$ in $\cat{Covering}_X$ to be non-surjective.

\section{Transitivity}\label{sec:transitivity}

In this section we reduce the study of special dagger Frobenius structures to the study of central ones and commutative ones, by proving a transitivity theorem that adapts~\cite[Theorem II.3.8]{demeyeringraham:separable} to the setting of dagger Frobenius structures. We start with combining Frobenius structures $E$ over $Z$ and $Z$ over $C$ into a Frobenius structure $E$ over $C$.

\begin{lemma}\label{lem:transitivity}
  Let $C$ and $Z$ be commutative C*-algebras with paracompact spectrum.
  If $E$ is a nondegenerate (specialisable) dagger Frobenius structure in $\cat{Hilb}_Z$, and $Z$ is a nondegenerate (specialisable) dagger Frobenius structure in $\cat{Hilb}_C$, 
  then $E$ is a nondegenerate (specialisable) dagger Frobenius structure in $\cat{Hilb}_C$.
\end{lemma}
\begin{proof}
  By Theorem~\ref{thm:takahashi:algebras}, there is a finite C*-bundle $p \colon E \twoheadrightarrow \Spec(Z)$, and a commutative finite C*-bundle $Z \twoheadrightarrow X=\Spec(C)$. By Theorem~\ref{thm:commutativecase}, the latter corresponds to a branched covering $q \colon \Spec(Z) \to X$. We will show that $r=q \circ p$ is a finite C*-bundle $E \twoheadrightarrow X$.
  First of all, the fibre of $r$ over $t \in X$ is $r^{-1}(t) = \bigoplus_{u \in q^{-1}(t)} p^{-1}(u)$, a finite direct sum of finite-dimensional C*-algebras, and hence a finite-dimensional C*-algebra.
  Now let $t_0 \in X$. Say $q^{-1}(t_0) = \{u_1,\ldots,u_n\} \in \Spec(Z)$.
  Pick open neighbourhoods $U_i \subseteq \Spec(Z)$ of $u_i$, finite-dimensional C*-algebras $A_i$, and homeomorphisms $\varphi_i \colon U_i \times A_i \to p^{-1}(U_i)$, such that $\varphi_i(u,-)\colon A_i \to p^{-1}(U_i)$ is a $*$-isomorphism for each $u \in U_i$.
  Because $q$ is a branched covering, we may assume the $U_i$ disjoint.
  Set $V=\bigcap_{i=1}^n q(V)$; this is an open neighbourhood of $t_0$ in $X$ because $q$ is open.
  Set $A=\bigoplus_{i=1}^n A_i$.
  Define $\varphi \colon V \times A \to r^{-1}(t_0) = \bigoplus_{i=1}^n p^{-1}(U_i)$ by
  \[
    \varphi(t,a) = \big( \varphi_1(u_1,a_1), \ldots, \varphi_n(u_n,a_n) \big)
  \]
  where $a=(a_1,\ldots,a_n)$, and $t=q(u_i)$ for $u_i \in U_i$.
  Then, for each $t \in V$, say $t=q(u_i)$ with $u_i \in U_i$, the function $$\varphi(t,-) = \bigoplus_{i=1}^n \varphi_i(u_i,(-)_i) \colon A = \bigoplus_{i=1}^n A_i \to \bigoplus_{i=1}^n p^{-1}(U_i) = r^{-1}(t)$$ is a $*$-isomorphism. 
  It is clear that $r$ is nondegenerate when $p$ and $q$ are, and that $r$ is specialisable when $p$ and $q$ are.
\end{proof}

The rest of this section considers the converse: if $E$ is a Frobenius structure over $C$, does it decompose into Frobenius structures $E$ over $Z$ and $Z$ over $C$? 
We start with the first step: $E$ over $Z$. Our proof below will use the following algebraic lemma.

\begin{lemma}\label{lem:hattori}
  If $(E,\mu,\eta)$ is a specialisable dagger Frobenius structure in $\cat{Hilb}_C$, then $E = Z(E) \oplus [E,E]$ is a dagger biproduct of Hilbert modules, where $[E,E]$ is the $C$-linear span of $\{xy-yx \mid x,y \in E\}$
\end{lemma}
\begin{proof}
  Adapting~\cite{aguiar:stronglyseparable} to monoidal categories, together with the fact that specialisable Frobenius structure are symmetric, shows that $(E,\mu,\eta)$ is strongly separable~\cite{demeyeringraham:separable}. 
  By~\cite[Theorem~1]{hattori:stronglyseparable}, there is a direct sum $E \simeq Z(E) \oplus [E,E]$ of $C$-modules.
  It now suffices to prove that this direct sum is orthogonal, as it then follows that both summands are Hilbert modules~\cite[Section~15.3]{weggeolsen:ktheory}.
  But if $z \in Z(E)$ and $x,y \in E$, then
  \[
    \inprod{z}{xy-yx} = \inprod{z}{xy} - \inprod{z}{yx} = \inprod{z y^*}{x} - \inprod{y^*z}{x} = 0,
  \]
  where the second equation uses that dagger Frobenius structures are H*-algebras; see~\cite[Lemma~5]{abramskyheunen:hstar}, which does not depend on commutativity. 
\end{proof}

It follows that the projection $p_1 \colon E \to Z(E)$ is cyclic: $p_1(xy)=p_1(yx)$.
It also follows that if $E$ is a specialisable dagger Frobenius structures, its centre $Z(E)$ is a well-defined Hilbert module.
We leave open the question whether special(isable) dagger Frobenius structures in arbitrary monoidal dagger categories correspond to monoid-comonoid pairs $E$ with $E \simeq Z(E) \oplus F$ a dagger biproduct, where $Z(E)$ is defined by an equaliser. 

Let us consider what the centre and commutator looks like in the paradigmatic example.

\begin{example}\label{ex:hattori}
  Consider the special dagger Frobenius structure $E=\M_n$ in $\cat{Hilb}$.
  Then $Z(E) = \C$, and $[E,E]=\{y \in \M_n \mid \tr(y)=0\}$ (see~\cite{albertmuckenhoupt:tracezero}) and indeed
  \[\begin{tikzpicture}
    \node (Z) at (-2,0) {$Z(E)$};
    \node (E) at (0,0) {$E$};
    \node (EE) at (2,0) {$[E,E]$};
    \draw[->] ([yshift=1mm]Z.east) to node[above]{$i_1$} ([yshift=1mm]E.west);
    \draw[->] ([yshift=-1mm]E.east) to node[below]{$p_2$} ([yshift=-1mm]EE.west);
    \draw[<-] ([yshift=-1mm]Z.east) to node[below]{$p_1$} ([yshift=-1mm]E.west);
    \draw[<-] ([yshift=1mm]E.east) to node[above]{$i_2$} ([yshift=1mm]EE.west);
  \end{tikzpicture}\]
  forms a dagger biproduct, where $i_1(1)=\tfrac{1}{\sqrt{n}}$, $i_2(y)=y$, $p_1(x)=\tfrac{1}{\sqrt{n}}\tr(x)$, and $p_2(x) = x-\tfrac{1}{n}x$:
  \begin{align*}
    \inprod{i_1(1)}{x} &= \inprod{\tfrac{1}{\sqrt{n}}}{x} = \tfrac{1}{\sqrt{n}}\tr(x) = \inprod{1}{p_1(x)}\text{,} \\
    \inprod{i_2(y)}{x} &= \inprod{y}{x} = \inprod{y}{x} - \inprod{y}{\tfrac{1}{n}\tr(x)} = \inprod{y}{p_2(x)}\text{,} \\
    p_1 \circ i_1 (1) &= p_1(\tfrac{1}{\sqrt{n}}) = \tfrac{1}{n} \tr(1) = 1\text{,} \\
    p_2 \circ i_2 (y) &= p_2(y) = y - \tfrac{1}{n}\tr(y) = y\text{,} \\
    i_1 \circ p_1 + i_2 \circ p_2 (x) &= i_1(\tfrac{1}{\sqrt{n}}\tr(x)) + i_2(x-\tfrac{1}{n}\tr(x)) = \tfrac{1}{n}\tr(x) + x - \tfrac{1}{n}\tr(x) = x\text{.}
  \end{align*}
\end{example}

Any special dagger Frobenius structure $E$ in $\cat{Hilb}_{C_0(X)}$ is a C*-algebra according to Lemma~\ref{lem:frobeniuscstar}. Therefore so is $Z(E)$, and it makes sense to talk about the monoidal category $\cat{Hilb}_{Z(E)}$.

The following two lemmas finish the proof of step one: if $E$ if Frobenius over $C$, then so is $E$ over $Z$. 

\begin{lemma}\label{lem:frobinCisfrobinZ}
  If $E$ is a special dagger Frobenius structure in $\cat{Hilb}_{C_0(X)}$, then it is also an object in $\cat{Hilb}_{Z(E)}$.
\end{lemma}
\begin{proof}
  First of all, $E$ is certainly a $Z(E)$-module; let us verify that it is a Hilbert $Z(E)$-module.
  As the inner product, take $\inprod{x}{y} = p_1(x^* y)$, using the projection $p_1 \colon E \to Z(E)$ induced by Lemma~\ref{lem:hattori}, and the involution~\eqref{eq:involution}.
  By Lemma~\ref{lem:hattori}, $p_1$ has norm one, and hence is a conditional expectation~\cite{tomiyama:conditionalexpectation}. Thus the inclusion $p_1^\dag \colon Z(E) \to E$ is a $*$-homomorphism, and $p_1$ is completely positive. 

  Because (completely) positive maps preserve the involution~\cite[p2]{stormer:positive}, we have $\inprod{y}{x}^* = p_1(y^*x)^* = p_1(x^*y) = \inprod{x}{y}$ for $x,y \in E$. Because $p_1$ is $Z(E)$-linear, also $\inprod{x}{y+y'}=\inprod{x}{y}+\inprod{x}{y'}$ and $\inprod{x}{yz}=\inprod{x}{y}z$ for $x,y,y' \in E$ and $z \in Z(E)$. Hence the inner product is $Z(E)$-sesquilinear.

  Again because $p_1$ is (completely) positive, $\inprod{x}{x} \geq 0$ for any $x \in E$. 
  To see that the inner product is in fact positive definite, first consider the case where $X=1$ and $E=\M_n$. Then $p_1 \colon \M_n \to \C^n$ takes the diagonal of a matrix. So if $x \in \M_n$, and $p_1(x^*x)=0$, then $x=0$, so certainly $p_1(x)=0$.
  This generalises to finite-dimensional C*-algebras $E$.
  Next we use Proposition~\ref{prop:localization} to go back to the case of general $E$: if $x \in E$ satisfies $p_1(x^*x)=0$, then for all $t \in X$ we have $\Loc_t(p_1(x))=0$. 
  So, by Theorem~\ref{thm:sectionsvanishingatinfinity}, in fact $p_1(x)=0$.
  Thus $\inprod{-}{-}$ is a well-defined $Z(E)$-valued inner product on $E$.

  The inner product is complete because
  \[
    \|x\|^2_{Z(E)} = \| \inprod{x}{x} \|_{Z(E)} = \| p_1(x^*x) \|_{Z(E)} \leq \|x^*x \|_{C_0(X)} = \|x\|^2_{C_0(X)}
  \]
  by Lemma~\ref{lem:frobeniuscstar}.
  Hence $E$ is a well-defined Hilbert $Z(E)$-module.
\end{proof}

\begin{lemma}\label{lem:frobeniusovercentre}
  If $E$ is a special dagger Frobenius structure in $\cat{Hilb}_{C_0(X)}$, then it is also a special dagger Frobenius structure in $\cat{Hilb}_{Z(E)}$.
\end{lemma}
\begin{proof}
  Write $C$ for $C_0(X)$ and $Z$ for Z(E).
  By definition, the tensor product of $E$ with itself in $\cat{Hilb}_{C}$, denoted $E \otimes_{C} E$, is the completion of the algebraic tensor product $E \odot_{C} E$ in the $C$-valued inner product $\inprod{x_1 \otimes y_1}{x_2 \otimes y_2} = \inprod{x_1}{x_2}\inprod{y_1}{y_2}$. Similarly, $E \otimes_{Z} E$ is the completion of $E \odot_{Z} E$ in the $Z$-valued inner product $\inprod{x_1 \otimes y_1}{x_2 \otimes y_2} = p_1(x_1^*x_2)p_1(y_1^*y_2)$.
  The assignment $x \otimes y \mapsto x \otimes y$ extends to a canonical map $q \colon E \otimes_{C} E \to E \otimes_{Z} E$, because if $x_i \in E \odot_{C} E$ converges in the former inner product, then it does so in the latter inner product too: 
  \[
    \| \inprod{x_i}{x_i}_E \|_{C}
    = \| x_i^* x_i \|_{C}
    \geq \| p_1(x_i^* x_i) \|_{Z}.
  \]
  Here, the equality uses that~\eqref{eq:involution} is a C*-involution locally as in Proposition~\ref{prop:localization}, and the inequality uses that $p_1$ has norm one.
  Because the multiplication $\mu$ is in fact $Z$-bilinear, it factors through $q$. This gives a map $\mu_Z$ that makes the following diagram of modules commute. 
  \[
    \begin{tikzpicture}[xscale=2.5]
      \node (l) at (0,0) {$E$};
      \node (t) at (1,1) {$E \otimes_{C} E$};
      \node (b) at (1,-1) {$E \otimes_{Z} E$};
      \node (tr) at (2,1) {$E \odot_{C} E$};
      \node (br) at (2,-1) {$E \odot_{Z} E$};
      \draw[->>] (t) to node[right]{$q$} (b);
      \draw[>->] (tr) to (t);
      \draw[>->] (br) to (b);
      \draw[->>] (tr) to (br);
      \draw[->] ([yshift=.5mm]t.south west) to node[above]{$\mu$} ([yshift=.5mm]l.north east);
      \draw[->, dashed] (b.north west) to node[below]{$\mu_Z$} (l.south east);
    \end{tikzpicture}
    \]
    Because $\mu^\dag(zx) = \mu^\dag \circ \mu (z \otimes x) = \mu_Z(z \otimes \mu^\dag(x)) = z \mu^\dag(x)$ by the Frobenius law and similarly $\mu^\dag(xz)=\mu^\dag(x)z$, the map $\mu^\dag \colon E \to E \otimes_{C} E$ is a morphism of $Z$-$Z$-bimodules. By construction $q$ is a map of $Z$-$Z$-bimodules. Hence $\mu_Z^\dag = q \circ \mu^\dag \colon E \to E \otimes_{Z} E$ is $Z$-linear.
    Now $\inprod{x \otimes y}{\mu_Z^\dag(w)}$ is computed as follows:
    \begin{align*}
      & \mu_Z \circ (p_1 \otimes_{C} p_1) \circ (\mu \otimes_C \mu) \circ (\id \otimes_C \sigma \otimes_C \id) \circ (\id \otimes_C \mu^\dag)(x^* \otimes y^* \otimes w) \\
      =\, & \mu_Z \circ (p_1 \otimes_C p_1) \circ (\id \otimes_C \mu) \circ (\sigma \otimes_C \id)\circ (\id \otimes_C \mu^\dag)  
      \circ (\id \otimes_C \mu)(x^* \otimes y^* \otimes w) \\
      =\, & \mu_Z \circ (p_1 \otimes_C \id) \circ (\mu \otimes_C p_1) \circ (\id \otimes_C \mu^\dag) \circ (\id \otimes_C \mu)(y^* \otimes w \otimes x^*) \\
      =\, & \mu_Z \circ \mu_Z^\dag \circ p_1 \circ \mu \circ (\id \otimes_C \mu)(y^* \otimes w \otimes x^*) \\
      =\, & \mu_Z \circ \mu_Z^\dag \circ p_1 \circ \mu \circ (\id\otimes_C \mu)(x^* \otimes w \otimes y^*)\text.
    \end{align*}
    This is perhaps easier to read graphically:
    \[
      \begin{pic}
        \node[dot] (b) at (1,0) {};
        \node[dot] (l) at (0,1) {};
        \node[dot] (r) at (1,1) {};
        \node[dot] (t) at (.5,1.5) {};
        \draw (b) to +(0,-.5) node[below]{$w\vphantom{^*}$};
        \draw[dashed] (l) to[out=90,in=180] (t);
        \draw[dashed] (t) to +(0,.5);
        \draw[dashed] (r) to[out=90,in=0] (t);
        \draw (b) to[out=0,in=0] (r);
        \draw (l) to[out=0,in=180] (b);
        \draw (l) to[out=180,in=90] (-.5,-.5) node[below]{$x^*$};
        \draw (r) to[out=180,in=90] (.25,-.5) node[below]{$y^*$};
      \end{pic}
      =
      \begin{pic}
        \node[dot] (b) at (0,0) {};
        \node[dot] (m) at (0,.5) {};
        \node[dot] (t) at (0,1) {};
        \node[dot] (z) at (-.25,1.5) {};
        \draw[dashed] (z) to +(0,.5);
        \draw[dashed] (z) to[out=0,in=90] (t);
        \draw[dashed] (z) to[out=180,in=180] (m);
        \draw (m) to[out=0,in=0] (t);
        \draw (b) to (m);
        \draw (b) to[out=0,in=90] +(.5,-.5) node[below]{$w\vphantom{^*}$};
        \draw (t) to[out=180,in=90] +(-.5,-1.5) node[below]{$x^*$};
        \draw (b) to[out=180,in=90] +(-1,-.5) node[below]{$y^*$};
      \end{pic}
      =
      \begin{pic}
        \node[dot] (b) at (0,.2) {};
        \node[dot] (m) at (0,.5) {};
        \node[dot] (t) at (-.5,1) {};
        \node[dot] (z) at (0,1.5) {};
        \draw[dashed] (z) to +(0,.5);
        \draw[dashed] (z) to[out=180,in=90] (t);
        \draw[dashed] (z) to[out=0,in=0] (m);
        \draw (t) to[out=0,in=180] (m);
        \draw (m) to (b);
        \draw (t) to[out=180,in=90] +(0,-1.5) node[below]{$y^*$};
        \draw (b) to[out=0,in=90] +(.2,-.1) to[out=-90,in=90] +(-1.2,-.6) node[below]{$x^*$};
        \draw (b) to[out=180,in=90] +(-.2,-.1) to[out=-90,in=90] +(.7,-.6) node[below]{$w\vphantom{^*}$};
      \end{pic}
      =
      \begin{pic}
        \node[dot] (b) at (.4,.3) {};
        \node[dot] (m) at (0,.6) {};
        \node[dot] (t) at (0,1) {};
        \node[dot] (z) at (0,1.5) {};
        \draw[dashed] (z) to +(0,.5);
        \draw[dashed] (z) to[out=180,in=180] (t);
        \draw[dashed] (z) to[out=0,in=0] (t);
        \draw (t) to (m);
        \draw (m) to[out=0,in=90] (b);
        \draw (m) to[out=180,in=90] +(0,-1) node[below]{$y^*$};
        \draw (b) to[out=0,in=90] +(.2,-.1) to[out=-90,in=90] +(-1,-.6) node[below]{$x^*$};
        \draw (b) to[out=180,in=90] +(-.2,-.1) to[out=-90,in=90] +(.4,-.6) node[below]{$w\vphantom{^*}$};
      \end{pic}
      =
      \begin{pic}
        \node[dot] (b) at (.4,.3) {};
        \node[dot] (m) at (0,.6) {};
        \node[dot] (t) at (0,1) {};
        \node[dot] (z) at (0,1.5) {};
        \draw[dashed] (z) to +(0,.5);
        \draw[dashed] (z) to[out=180,in=180] (t);
        \draw[dashed] (z) to[out=0,in=0] (t);
        \draw (t) to (m);
        \draw (m) to[out=0,in=90] (b);
        \draw (m) to[out=180,in=90] +(-.5,-1) node[below]{$x^*$};
        \draw (b) to[out=180,in=90] +(-.2,-.2) to[out=-90,in=90] +(.6,-.5) node[below]{$w\vphantom{^*}$};
        \draw (b) to[out=0,in=90] +(.2,-.2) to[out=-90,in=90] +(-.6,-.5) node[below]{$y^*$};
      \end{pic}
    \]
    where we draw solid lines for $E$ and dashed lines for $Z$; the first and third equalities use the strong Frobenius law, and the second and fourth equalities use associativity, naturality of the swap map, and the fact that $Z$ is commutative. 
    Thus
    \[
      \inprod{x \otimes y}{\mu_Z^\dag(w)} 
      = p_1 \circ \mu_Z \circ (p_1^\dag \otimes_{Z} p_1^\dag) \circ (p_1 \otimes_{Z} p_1) \circ \mu_Z^\dag (y^* x^* w)
    \]
    because cyclicity of $p_1$ allows us to change $x^*wy^*$ into $y^*x^*w$ under this map.
    On the other hand, $\inprod{\mu_Z(x \otimes y)}{w}$ is $p_1(y^* x^* w)$.
    Because $(p_1 \otimes_{Z} p_1) \circ \mu_Z^\dag$ is an isometry, 
    $p_1 = p_1 \circ \mu_Z \circ (p_1^\dag \otimes_{Z(E)} p_1^\dag) \circ (p_1 \otimes_{Z(E)} p_1) \circ \mu_Z^\dag$.
    Therefore $\mu_Z$ and $\mu_Z^\dag$ are adjoints.
    
  We can now verify the laws for special dagger Frobenius structures for $\mu_Z$.
  Unitality of $\mu_Z$ follows directly from unitality of $\mu$ because $\eta$ factors through $Z$.
  Speciality is also easy: $\mu_Z \circ \mu_Z^\dag = \mu_Z \circ q \circ \mu^\dag = \mu \circ \mu^\dag = \id[E]$.
  Now observe that $q \circ (\mu_Z \otimes_{C} \id[E]) = (\mu_Z \otimes_{Z} \id[E]) \circ (\id[E] \otimes_{Z} q)$, because both morphisms map $x \otimes y \otimes z$ to $\mu_Z(x \otimes y) \otimes z$.
  It follows from associativity of $\mu$ that 
  \begin{align*}
    & \mu_Z \circ (\mu_Z \otimes_{Z} \id[E]) \circ (\id[E] \otimes_{Z} q) \circ (q \otimes_{C} \id[E]) \\
    = \;& \mu_Z \circ (\id[E] \otimes_{Z} \mu_Z) \circ (q \otimes_{Z} \id[E]) \circ (\id[E] \otimes_{C} q)\text{.}
  \end{align*}
  Therefore $\mu_Z \circ (\id[E] \otimes_{Z} \mu_Z)$ equals $\mu_Z \circ (\mu_Z \otimes_{Z} \id[E])$ on $E \odot_{Z} E \odot_{Z} E$ and hence on all of $E \otimes_{Z} E \otimes_{Z} E$, making $\mu_Z$ associative.
  The Frobenius law follows similarly: the two morphisms
  \begin{align*}
    (\mu_Z \otimes_{Z} \id[E]) \circ (\id[E] \otimes_{Z} \mu_Z^\dag) 
    & = q \circ (\mu_Z \otimes_{C} \id[E]) \circ (\id[E] \otimes_{Z} \mu^\dag) \\
    (\id[E] \otimes_{Z} \mu_Z) \circ (\mu_Z^\dag \otimes_{Z} \id[E]) 
    & = q \circ (\id[E] \otimes_{C} \mu_Z) \circ (\mu^\dag \otimes_{Z} \id[E])
  \end{align*}
  equal each other on $E \odot_{Z} E$, and are therefore equal on all of $E \otimes_{Z} E$.
\end{proof}

The last step is to prove that if $E$ is Frobenius over $C$, then so is its centre $Z(E)$.

\begin{lemma}\label{lem:centerfrobenius}
  Let $C$ be a commutative C*-algebra with a paracompact spectrum.
  If $E$ is a special dagger Frobenius structure in $\cat{Hilb}_C$, then $Z(E)$ is a specialisable dagger Frobenius structure in $\cat{Hilb}_C$. 
\end{lemma}
\begin{proof}
  By Theorem~\ref{thm:takahashi:algebras}, $E$ corresponds to a finite C*-bundle $p \colon E \twoheadrightarrow X$.
  Define $q \colon Z(E) \to X$ by restriction; we will prove that it is a commutative finite C*-bundle.
  Clearly, $q$ is still continuous and surjective, because it maps $1 \in Z(E_t)$ to $t \in X$.
  Also, $Z(E)_t = Z(E_t)$ is a commutative finite-dimensional C*-algebra.
  Now let $t_0 \in X$. Pick an open neighbourhood $U$ of $t_0$ in $X$, a finite-dimensional C*-algebra $A$, and a map $\varphi \colon U \times A \to p^{-1}(U)$ such that $\varphi(t,-)\colon  \to E_t$ is a $*$-isomorphism for every $t \in U$.
  Set $B=Z(A)$, and define $\psi \colon U \times B \to q^{-1}(U) = p^{-1}(U) \cap Z(E)$ to be the restriction of $\varphi$.
  Then $\psi(t,-) \colon B \to q^{-1}(t) = Z(E)_t$ is a $*$-isomorphism.
\end{proof}

Finally, we can state the transitivity theorem.

\begin{theorem}\label{thm:transitivity}
  Let $X$ be a paracompact locally compact Hausdorff space, and $E$ a monoid in $\cat{Hilb}_{C_0(X)}$.
  The following are equivalent:
  \begin{enumerate}[label=(\roman*)]
  \item $E$ is a special dagger Frobenius structure in $\cat{Hilb}_{C_0(X)}$;
  \item $E$ is a special dagger Frobenius structure in $\cat{Hilb}_{Z(E)}$, and \\$Z(E)$ is a specialisable dagger Frobenius structure in $\cat{Hilb}_{C_0(X)}$.
  \end{enumerate}
\end{theorem}
\begin{proof}
  Combine Lemmas~\ref{lem:transitivity}, \ref{lem:frobeniusovercentre}, and~\ref{lem:centerfrobenius}. The only thing left to prove is that $E$ is special over $Z(E)$ precisely when it special over $C_0(X)$. But this is already included in the proof of Lemma~\ref{lem:frobeniusovercentre}. 
\end{proof}

The latter algebra in (ii) is commutative, the former is central.
We leave open the question to which monoidal dagger categories the previous theorem can be generalised~\cite{heunenvicarywester:two}; there needs to be enough structure to make sense of the centre of a monoid. We also leave open the question whether it can be made functorial, that is, how the categories and Frobenius structures in (ii) of the previous theorem depend on $E$ and $X$.

\section{Kernels}\label{sec:kernels}

In this final section, we return to the question of Section~\ref{sec:scalars}: what can be said about the base space given just the category of Hilbert modules? We will study special kinds of maps into the tensor unit, namely kernels. It will turn out that the existence of kernels is related to clopen subsets disconnectedness properties of the base space.

A dagger category with a zero object has \emph{dagger kernels} when every morphism $f \colon E \to F$ has a kernel $k \colon K \to E$ satisfying $k^\dag \circ k =\id[E]$~\cite{heunenjacobs:kernels}. 
Similarly, it has \emph{dagger equalisers} when every pair of morphisms $f,g \colon E \to F$ has an equaliser $e$ satisfying $e^\dag \circ e = \id$.
In this section we show that $\cat{FHilb}_{C_0(X)}$ has dagger kernels, and discuss when $\cat{Hilb}_{C_0(X)}$ has dagger kernels.

\begin{proposition}\label{prop:kernels}
  If $X$ is a locally compact Hausdorff space, $\cat{Hilb}_{C_0(X)}\bd$ has kernels; the kernel of $f \colon E \to F$ is given by (the inclusion of) $\ker(f) = \{x \in E \mid f(x)=0\}$. 
\end{proposition}
\begin{proof}
  We prove that $\ker(f)$ is always a well-defined object in $\cat{Hilb}_{C_0(X)}\bd$.
  The inherited inner product $\inprod{x}{y}_K = \inprod{x}{y}_E$ is still sesquilinear and positive definite. If $(x_n)$ is a Cauchy sequence in $\ker(f)$, it is also a Cauchy sequence in $E$, and hence has a limit $x \in E$. Because $f$ is adjointable, it is bounded and hence continuous, so that $f(x)=\lim_n f(x_n) = 0$ and $x \in \ker(f)$. Thus $\ker(f)$ is complete.

  The inclusion $\ker(f) \hookrightarrow E$ is bounded because it is fibrewise contractive, and hence a well-defined morphism. It inherits the universal property from the category of vector spaces.
\end{proof}

\begin{proposition}\label{prop:fhilbdaggerkernels}
  If $X$ is a paracompact locally compact Hausdorff space, then $\cat{FHilb}_{C_0(X)}$ has dagger kernels; the kernel of $f \colon E \to F$ is given by (the inclusion of) $\ker(f) = \{x \in E \mid f(x)=0\}$. 
\end{proposition}
\begin{proof}
  First, notice that $K=\ker(f)$ is indeed a well-defined object of $\cat{FHilb}_{C_0(X)}$ by Theorem~\ref{thm:takahashi}: for a subbundle $\ker(f)$ of a finite Hilbert bundle $E$ is a finite Hilbert bundle.
  By Theorem~\ref{thm:dualobjects}, this means there exists $L \in \cat{FHilb}_{C_0(X)}$ such that $K \oplus L \simeq C_0(X)^m$ for some natural number $m$.
  Next, because the map $t \mapsto \dim(E_t)$ is continuous, we can write $X$ as a disjoint union of clopen subsets on which the fibres of $E$ and $F$ have constant dimension. Thus we may assume that $E=C_0(X)^n$ for some natural number $n$.
  Now the inclusion $k \colon K \to E$ is adjointable if and only if the map $[k,0] \colon K \oplus L \simeq C_0(X)^m \to C_0(X)^n$ is. 
  But this follows from Lemma~\ref{lem:finitelypresentedmapsadjointable} because $k$ is bounded.
\end{proof}

When we consider Hilbert modules that are not necessarily finitely presented projective, dagger kernels do not always exist. If they do, the base space $X$ must be \emph{totally disconnected}, that is, its connected components must be singletons. 
If $X$ is compact this is equivalent to $C(X)$ being a C*-algebra of \emph{real rank zero}.

\begin{proposition}\label{prop:adjointablekernels}
  Let $X$ be a locally compact Hausdorff space.
  If $\cat{Hilb}_{C_0(X)}$ has dagger kernels, then $X$ is totally disconnected.
\end{proposition}
\begin{proof}
  Let $U \subseteq X$ be a closed set containing distinct points $x,y \in X$.
  Since $X$ is Hausdorff, $x$ and $y$ have disjoint open neighbourhoods $V_x$ and $V_y$.
  Now $\{y\}$ is compact and $V_y$ is open, so Urysohn's lemma constructs $f \in C_0(X)$ with $f(y)=1$ and $f(X \setminus V_y)=0$ so $f(x)=0$.
  Regard $f$ as a morphism $C_0(X) \to C_0(X)$ by $h \mapsto fh$; it has adjoint $h \mapsto f^* h$.
  As in Lemma~\ref{lem:daggersubobjects}, $f$ has a dagger kernel of the form $K = \{ h \in C_0(X) \mid h(W)=0\}$ for a clopen $W \subseteq X$.
  Now $U_x = U \cap (X \setminus W)$ and $U_y = U \cap W$ are both open in $U$, satisfy $U=U_x \cup U_y$ and $U_x \cap U_y = \emptyset$, and are not empty because $x \in U_x$ and $y \in U_y$. 
  Therefore $U$ is not connected.
  That is, $X$ is totally disconnected.
\end{proof}

\begin{remark}
  If $X$ is totally disconnected, does $\cat{Hilb}_{C_0(X)}$ have dagger kernels? 
  The question is whether the inclusion $\ker(f) \hookrightarrow E$ is adjointable. 
  The luxury of finitely presented projectivity as used in the proof of Proposition~\ref{prop:fhilbdaggerkernels} is not available.
  In general it would suffice for $\ker(f)$ to be \emph{self-dual}~\cite[3.3-3.4]{paschke:selfdual}, but it is unclear whether $\ker(f)$ is self-dual when $E$ and $F$ are self-dual and $X$ is totally disconnected; for related functional-analytic problems see~\cite{frankpaulsen:injective,frank:monotonecomplete}. 
  We leave this question open.
\end{remark}

\begin{remark}
  Which categories $\cat{C}$ embed into $\cat{FHilb}_{C_0(X)}$ or $\cat{Hilb}_{C_0(X)}$ for some $X$? 
  We might generalise the strategy of~\cite[7.2]{heunen:embedding} that worked for $\cat{Hilb}$ while removing an inelegant cardinality restriction on the scalars:
  it suffices that $\cat{C}$ is symmetric dagger monoidal; has finite dagger biproducts; has dagger equalisers of cotuples $[f,g],[g,f] \colon E \oplus E \to F$ for $f,g \colon E \to F$; makes every dagger monomorphism a dagger kernel; is well-pointed, and is locally small. 
  The scalars $\cat{C}(I,I)$ then form a unital commutative $*$-ring, and we would need an additional condition guaranteeing that it embeds into a commutative complex *-algebra $C_b(X)$ for some $X$.
  To embed into $\cat{FHilb}_{C_0(X)}$, we additionally require every object in the category $\cat{C}$ to have a dagger dual object.
  As a sanity check that these properties do indeed characterise categories $\cat{C}$ embedding into $\cat{FHilb}_{C_0(X)}$ for some $X$, note that the category $\cat{FHilb}_{C_0(X)}$ itself satisfies all of these properties~\cite[3.6]{lance:hilbert}. 
\end{remark}

\appendix
\section{Bimodules and bicategories}\label{sec:bimodules}

This appendix considers the structure that results if we let the base space vary. Instead of Hilbert modules we then need to consider Hilbert bimodules, and we end up with a bicategory. It will turn out that this bicategory is a continuous extension of the well-studied bicategory of 2-Hilbert spaces.

We start by briefly recalling Hilbert bimodules and their tensor products; for more information we refer to~\cite{lance:hilbert}.
Recall that the adjointable maps $E \to E$ on a Hilbert $A$-module $E$ form a C*-algebra $\L(E)$.

\begin{definition}
  Let $A$ and $B$ be C*-algebras. A \emph{Hilbert $(A,B)$-bimodule} is a (right) Hilbert $B$-module $E$ together with a $*$-homomorphism $\varphi \colon A \to \L(E)$ that is nondegenerate, in the sense that $\varphi(A)(E)$ is dense in $E$.
  A \emph{morphism of Hilbert $(A,B)$-bimodules} is an adjointable map $f \colon E \to F$ of (right) Hilbert $B$-modules that intertwines, \textit{i.e.}\ $f(a(x))=a(f(x))$ for $a \in A$ and $x \in E$.
\end{definition}

A Hilbert $\C$-module is simply a Hilbert space, and a morphism of $\C$-modules is simply an adjointable map between Hilbert spaces.
A Hilbert $A$-module is the same as a Hilbert $(\C,A)$-bimodule, and a morphism of Hilbert $(\C,A)$-bimodules is the same as an adjointable map of Hilbert $A$-modules.
Hence a Hilbert $(A,\C)$-bimodule is precisely a $*$-representation of $A$, and a morphism of Hilbert $(A,\C)$-bimodules is precisely an intertwiner.

\begin{definition}\label{def:tensorproduct}
  The \emph{tensor product} $E \otimes_B F$ of a Hilbert $(A,B)$-bimodule $E$ and a Hilbert $(B,C)$-bimodule $F$ is the algebraic tensor product of $\C$-modules $E \otimes_{\C} F$ made into a Hilbert $A$-$C$-bimodule under the inner product
  \[
    \inprod{x \otimes y}{x' \otimes y'}_{E \otimes_{\C} F} = \inprod{y}{ \inprod{x}{x'}_E(y') }_F
  \]
  by quotienting out $\{x \in E \otimes_{\C} F \mid \inprod{x}{x}_{E \otimes_{\C} F} = 0 \}$ and completing, with the map $A \to \L(E \otimes_B F)$ sending $a$ to $x \otimes y \mapsto a(x) \otimes y$.

  Notice that this quotient automatically enforces $xb \otimes y = x \otimes by$ in $E \otimes_B F$ for $x \in E$, $y \in F$, and $b \in B$. So $E \otimes_B F$ may alternatively be constructed as the algebraic tensor product $E \odot_B F$ over $B$ of $A$-$B$-bimodules and $B$-$C$-bimodules by quotienting out the same subspace and completing in the same inner product.
\end{definition}

The tensor product $E \otimes F$ of Hilbert $A$-modules $E$ and $F$ over a commutative $A$ is got by regarding them as Hilbert $(\C,A)$-bimodules. 
If $A$ is commutative, $F$ is also a Hilbert $(A,A)$-bimodule, via the map $A \to \L(F)$ that sends $a$ to $y \mapsto ya$. 
The tensor product $E \otimes_A F$ of Hilbert bimodules then is a Hilbert $(\C,A)$-bimodule and hence a Hilbert $A$-module $E \otimes F$. Explicitly, it is the completion of the algebraic tensor product $E \otimes_{\C} F$ with the following inner product and (right) $A$-module structure:
\begin{align*}
  \inprod{x_1 \otimes y_1}{x_2 \otimes y_2} & = \inprod{x_1}{x_2} \inprod{y_1}{y_2}, \\
  (x \otimes y) a & = x \otimes (ya).
\end{align*}
Note that this inner product is indeed already nondegenerate~\cite[Proposition~4.5]{lance:hilbert}.

If $f \colon E \to E'$ is a morphism of Hilbert $A$-$B$-bimodules, and $g \colon F \to F'$ is a morphism of Hilbert $B$-$C$-bimodules, then the canonical map $f \otimes_B g \colon E \otimes_B F \to E' \otimes_B F'$ defined by $x \otimes y \mapsto f(x) \otimes g(y)$ for $x \in E$ and $y \in F$ is a well-defined morphism of Hilbert $A$-$C$-bimodules: it is adjointable because $g$ is an intertwiner
\begin{align*}
  \inprod{x' \otimes y'}{f \otimes g(x \otimes y)}_{E' \otimes_B F'}
  & = \inprod{x' \otimes y'}{f(x) \otimes g(y)}_{E' \otimes_B F'} \\
  & = \inprod{y'}{ \inprod{x'}{f(x)}_{E'} (g(y))}_{F'} \\
  & = \inprod{y'}{ g( \inprod{x'}{f(x)}_{E'}(y))}_{F'} \\
  & = \inprod{g^\dag(y')}{ \inprod{f^\dag(x')}{x}_E(y)}_F \\
  & = \inprod{f^\dag \otimes g^\dag (x' \otimes y')}{x \otimes y}_{E \otimes_B F},
\end{align*}
and it is an intertwiner because $f$ is an intertwiner
\[
  a(f \otimes g(x))
  = a(f(x)) \otimes g(y)
  = f(a(x)) \otimes g(y)
  = f \otimes g(a(x)).
\]

\begin{proposition}
  There is a well-defined bicategory $\cat{Hilb}_*$ where:
  \begin{itemize}
  \item 0-cells are locally compact Hausdorff spaces $X$;
  \item 1-cells are Hilbert $C_0(X)$-$C_0(Y)$-bimodules;
  \item the identity 1-cell on $X$ is $C_0(X)$;
  \item horizontal composition of 1-cells is $(E,F) \mapsto E \otimes_{C_0(Y)} F$;
  \item 2-cells are morphisms of Hilbert $C_0(X)$-$C_0(Y)$-bimodules, \textit{i.e.}\ adjointable intertwiners;
  \item the identity 2-cell on $E$ is the identity function;
  \item vertical composition of 2-cells is function composition;
  \item horizontal composition of 2-cells is $(f,g) \mapsto f \otimes_{C_0(Y)} g$;
  \item associators $(E \otimes_{C_0(Y)} F) \otimes_{C_0(Z)} G \to E \otimes_{C_0(Y)} (F \otimes_{C_0(Z)} G)$ are given by $(x \otimes y) \otimes z \mapsto x \otimes (y \otimes z)$;
  \item left unitors $C_0(X) \otimes_{C_0(X)} E \to E$ are given by $a \otimes x \mapsto a(x)$;
  \item right unitors $E \otimes_{C_0(Y)} C_0(Y) \to E$ are given by $x \otimes b \mapsto xb$;
  \end{itemize}
  as well as a bicategory $\cat{Hilb}_*\bd$ where 2-cells are bounded linear intertwiners.
\end{proposition}
\begin{proof}
  We have already seen that the homcategories are well-defined, and that horizontal composition is a well-defined functor. The pentagon equations are clear. The triangle equations $(\id[E] \otimes_{C_0(Y)} \lambda_F) \circ \alpha_{E,C_0(Y),F} = \rho_E \otimes_{C_0(Y)} \id[F]$ are satisfied because $xb \otimes_{C_0(Y)} y = x \otimes_{C_0(Y)} b(y)$ for $b \in C_0(Y)$, $x \in E$, and $y \in F$.
  See also~\cite{busszhumeyer:highercstar}, who use a stronger notion of 2-cells.
\end{proof}

Notice that the endohomcategory $\cat{Hilb}_*(X,X)$ equals $\cat{Hilb}_{C_0(X)}$, so that the first (non-dagger) half of Proposition~\ref{prop:monoidal} follows from the previous one.

There is also a well-defined bicategory $\cat{2FHilb}$ of \emph{2-Hilbert spaces}, which has as 0-cells natural numbers, as 1-cells matrices of finite-dimensional Hilbert spaces, and as 2-cells matrices of linear maps~\cite{vicary:higher,heunenvicarywester:two}. 

\begin{proposition}
  There is a pseudofunctor $\cat{2FHilb} \to \cat{Hilb}_*$ that:
  \begin{itemize}
  \item sends a 0-cell $n$ to $\{1,\ldots,n\}$;
  \item sends a 1-cell $(H_{i,j}) \colon m \to n$ to $\bigoplus_{i,j} H_{i,j}$;
  \item sends a 2-cell $(f_{i,j}) \colon (H_{i,j})\to (K_{i,j})$ to the map $(x_{i,j}) \mapsto (f_{i,j}(x_{i,j}))$;
  \item is injective on 0-cells, and a local equivalence.
  \end{itemize}
\end{proposition}
\begin{proof}
  Let us show that this is well-defined on 1-cells:
  $E=\bigoplus_{i,j} H_{i,j}$ becomes a right $\C^n$-module by $(x_{i,j}) \cdot (z_j) = (x_{ij}z_j)$;
  it becomes a (right) Hilbert $\C^n$-module by the inner product $\inprod{(x_{i,j})}{(y_{i,j})}_E(j) = \sum_i \inprod{x_{i,j}}{y_{i,j}}_{H_{i,j}}$;
  it becomes a Hilbert $\C^m$-$\C^n$-bimodule by the $*$-representation $\C^m \to \L(E)$ sending $(z_i)$ to $(x_{i,j}) \mapsto (z_i x_{i,j})$. 

  It is also well-defined on 2-cells: the map $f \colon x_{i,j} \mapsto (f_{i,j}(x_{i,j}))$ is adjointable because $\sum_i \inprod{f_{i,j}(x_{i,j})}{y_{i,j}}_{K_{i,j}} = \sum_i \inprod{x_{i,j}}{f^\dag_{i,j}(y_{i,j})}_{H_{i,j}}$; and it is intertwining because $f_{i,j}(z_i x_{i,j}) = z_i f_{i,j}(x_{i,j})$.
  This is clearly functorial on homcategories.

  The pseudofunctorial data consists of 2-cells $\C^n \to \bigoplus_{i,j=1}^n \delta_{i,j} \C$ for identities, and $(\bigoplus_{a,b} H_{a,b}) \otimes_{\C^n} (\bigoplus_{c,d} K_{c,d}) \to \bigoplus_{i,j,k} H_{i,k} \otimes K_{k,j}$ for composition.
  By construction $(\bigoplus_{a,b}H_{a,b}) \otimes_{\C^n} (\bigoplus_{c,d} K_{c,d})$ is $\bigoplus_{a,b,c,d} H_{a,b} \otimes K_{c,d}$, where we identify $((x_{a,b}) \otimes (y_{c,d}))$ with 0 when $x_{a,b} y_{b,d}=0$ for all $a$ and $d$.
  Hence there are natural candidates for both, that are adjointable intertwiners, and furthermore are in fact unitary.
  The coherence diagrams clearly commute.

  Finally, this pseudofunctor is clearly injective on 0-cells, and moreover, it is an equivalence on homcategories; see also~\cite[Proposition~8.1.11]{blecherlemerdy:modules}.
\end{proof}

Thus $\cat{2FHilb}$ is a full subcategory of $\cat{Hilb}_*$. In other words, $\cat{Hilb}_*$ is a conservative \emph{infinite continuous extension} of the finite discrete $\cat{2FHilb}$ that is more suitable for local quantum physics.

\section{Complete positivity}\label{sec:radon}

Localization, as discussed in Section~\ref{sec:tensor}, is essential to the theory of Hilbert modules. And conditional expectations are essential to localization. They are a certain kind of completely positive map. In this appendix we study the category of commutative C*-algebras and completely positive maps further.
Write $\cat{cCstar}_\cp$ for the category of commutative C*-algebras and (completely) positive linear maps.
By Gelfand duality, its objects are isomorphic to $C_0(X)$ for locally compact Hausdorff spaces $X$.
We now consider morphisms.

\begin{definition}
  A \emph{Radon measure} on a locally compact Hausdorff space $X$ is a positive Borel measure $\mu$ satisfying $\mu(U) = \sup_{K \subseteq U} \mu(K)$ where $K$ ranges over the compact subsets of open sets $U$.
  Write $\Radon(X)$ for the set of Radon measures on $X$.
\end{definition}

The set $\Radon(X)$ becomes a locally compact Hausdorff space~\cite[Chapter~13]{taylor:measure} under the following, so-called \emph{vague}, topology: a net $\mu_n$ converges to $\mu$ if and only if $\int_X f \,\mathrm{d}\mu_n$ converges to $\int_X f\,\mathrm{d}\mu$ for all measurable $f \colon X \to \C$.

\begin{definition}
  Write $\cat{Radon}$ for the following category.
  \begin{itemize}
  \item Objects are locally compact Hausdorff spaces $X$.
  \item Morphisms $X \to Y$ are continuous functions $X \to \Radon(Y)$.
  \item Composition of $f \colon X \to \Radon(Y)$ and $g \colon Y \to \Radon(Z)$ is given by
  \[
    (g \circ f)(x)(U) = \int_Y g_U \,\mathrm{d}f(x)
  \]
  where $g_U \colon Y \to \C$ for measurable $U \subseteq Z$ is defined by $y \mapsto g(y)(U)$.
  \item The identity on $X$ sends $x$ to the Dirac measure $\delta_x$.
  \end{itemize}
\end{definition}

\begin{proposition}\label{prop:radon}
  There is an equivalence of categories
  \begin{align*}
    F \colon \cat{Radon} & \to \cat{cCstar}_\cp\op \\
    F(X) & = C_0(X) \\
    F(f)(h)(x) & = \int_X h \,\mathrm{d}f(x)\text{.}
  \end{align*}
\end{proposition}
\begin{proof}
  The proof of~\cite[Theorem~5.1]{furberjacobs:gelfand} shows that $F(X)=C_0(X)$ and $F(f)(h)(x) = f(x)(h)$ define an equivalence $F \colon \cat{R} \to \cat{Cstar}_\cp\op$, for the following category $\cat{R}$:
  \begin{itemize}
  \item Objects are locally compact Hausdorff spaces $X$.
  \item Morphisms $X \to Y$ are continuous maps $X \to R(Y)=\cat{cCstar}_\cp(C_0(Y),\C)$.
  \item Composition of $f \colon X \to R(Y)$ and $g \colon Y \to R(Z)$ is given by
    \[
      (g \circ f)(x)(\varphi) = f(x)(\ev_\varphi \circ g)
    \]
    where $\ev_\varphi \colon R(Z) \to \C$ for $\varphi \in C_0(Z)$ is defined by $\ev_\varphi(h)=h(\varphi)$.
  \item The identity on $X$ sends $x$ to the map $C_0(X) \to \C$ defined by $k \mapsto k(x)$.
  \end{itemize}
  But every element of $R(X)$ is of the form $\int_X (-) \,\mathrm{d}\mu$ for a unique $\mu \in \Radon(X)$ (see~\cite[Theorem~2.14]{rudin:analysis}), translating to the statement of the proposition.
\end{proof}

Finally we consider the special case of conditional expectations.

\begin{proposition}\label{prop:conditionalexpectations}
  The wide subcategory $\cat{Cstar}_\cp$ of conditional expectations is dually equivalent to the wide subcategory $\cat{Radon}_\cp$ of $\cat{Radon}$ of morphisms $f \colon X \to \Radon(Y)$ with a continuous surjection $g \colon Y \twoheadrightarrow X$ satisfying $\supp(f(x)) \subseteq g^{-1}(x)$.
\end{proposition}
\begin{proof}
  Simply restrict the equivalence of Proposition~\ref{prop:radon}.
  Concretely, a morphism $(f,g)$ of $\Radon_\cp$ gets sent to the following conditional expectation: the injective $*$-homomorphism is $- \circ g \colon C_0(X) \rightarrowtail C_0(Y)$, and the completely positive map $C_0(Y) \twoheadrightarrow C_0(X)$ maps $\varphi \in C(Y)$ to the function $x \mapsto \int_Y g \,\mathrm{d}f(x)$.
  Conversely, a conditional expectation $E$ is sent to the unique morphism $(f,g)$ satisfying $E(\varphi)(y) = \int_X \varphi \, \mathrm{d} g(f(y))$.
  See also~\cite[Theorem~5.3.3]{pluta:corners}.
\end{proof}

\bibliographystyle{plain}
\bibliography{frobenius}

\begin{thebibliography}{10}

\bibitem{abramskybrandenburger:contextuality}
S.~Abramsky and A.~Brandenburger.
\newblock The sheaf-theoretic structure of non-locality and contextuality.
\newblock {\em New Journal of Physics}, 13:113036, 2011.

\bibitem{abramskyheunen:hstar}
S.~Abramsky and C.~Heunen.
\newblock H*-algebras and nonunital {F}robenius algebras: first steps in
  infinite-dimensional categorical quantum mechanics.
\newblock {\em Clifford Lectures, AMS Proceedings of Symposia in Applied
  Mathematics}, 71:1--24, 2012.

\bibitem{abramskyheunen:operational}
S.~Abramsky and C.~Heunen.
\newblock {\em Logic and algebraic structures in quantum computing and
  information}, chapter Operational theories and categorical quantum mechanics.
\newblock Cambridge University Press, 2015.

\bibitem{aguiar:stronglyseparable}
M.~Aguiar.
\newblock A note on strongly separable algebras.
\newblock {\em Boletin de la Academia Nacional de Ciencias}, 65:51--60, 2000.

\bibitem{albertmuckenhoupt:tracezero}
A.~A. Albert and B.~Muckenhoupt.
\newblock On matrices of trace zero.
\newblock {\em Mich. Math. J.}, 4(1):1--3, 1957.

\bibitem{antonevichkrupnik:trivial}
A.~Antonevich and N.~Krupnik.
\newblock On trivial and non-trivial $n$-homogeneous {C}*-algebras.
\newblock {\em Integr. equ. oper. theory}, 38:172--189, 2000.

\bibitem{auslandergoldman:brauer}
M.~Auslander and O.~Goldman.
\newblock The {B}rauer group of a commutative ring.
\newblock {\em Transactions of the American Mathematical Society}, 97:367--409,
  1960.

\bibitem{blackadar:operatoralgebra}
B.~Blackadar.
\newblock {\em Operator algebras: theory of {C}*-algebras and von {N}eumann
  algebras}.
\newblock Springer, 2006.

\bibitem{blanchardkirchberg:glimmhalving}
E.~Blanchard and E.~Kirchberg.
\newblock Global {G}limm halving for {C}*-bundles.
\newblock {\em Journal of Operator Theory}, 52:385--420, 2004.

\bibitem{blecherlemerdy:modules}
D.~P. Blecher and C.~{Le Merdy}.
\newblock {\em Operator algebras and their modules, an operator space
  approach}.
\newblock Oxford University Press, 2004.

\bibitem{blutecomeau:neumann}
R.~Blute and M.~Comeau.
\newblock Von {N}eumann categories.
\newblock {\em Applied Categorical Structures}, 23(5):725--740, 2015.

\bibitem{bos:groupoids}
R.~Bos.
\newblock Continuous representations of groupoids.
\newblock {\em Houston Journal of Mathematics}, 37(3):807--844, 2011.

\bibitem{busszhumeyer:highercstar}
A.~Buss, C.~Zhu, and R.~Meyer.
\newblock A higher category approach to twisted actions on {C}*-algebras.
\newblock {\em Proceedings of the Edinburgh Mathematical Society}, 56:387--426,
  2013.

\bibitem{coeckeheunenkissinger:cpstar}
B.~Coecke, C.~Heunen, and A.~Kissinger.
\newblock Categories of quantum and classical channels.
\newblock {\em Quantum Information Processing}, 2014.

\bibitem{coeckelal:causal}
B.~Coecke and R.~Lal.
\newblock Causal categories: relativistically interacting processes.
\newblock {\em Foundations of Physics}, 43(4):458--501, 2012.

\bibitem{daunshofmann:sections}
J.~Dauns and K.~H. Hofmann.
\newblock {\em Representation of rings by sections}.
\newblock American Mathematical Society, 1968.

\bibitem{demeyeringraham:separable}
F.~{DeMeyer} and E.~Ingraham.
\newblock {\em Separable algebras over commutative rings}.
\newblock Number 181 in Lecture Notes in Mathematics. Springer, 1971.

\bibitem{dixmier:cstaralgebras}
J.~Dixmier.
\newblock {\em C*-algebras}.
\newblock North Holland, 1981.

\bibitem{dixmierdouady:champs}
J.~Dixmier and A.~Douady.
\newblock Champs continues d'espaces {H}ilbertiens et de {C}*-alg{\`e}bres.
\newblock {\em Bulletin de la Soci{\'e}t{\'e} Math{\'e}matique de France},
  91:227--284, 1963.

\bibitem{dupre:classifyinghilbertbundles}
M.~J. Dupr{\'e}.
\newblock Classifying {H}ilbert bundles.
\newblock {\em Journal of Functional Analysis}, 15(3):244--278, 1974.

\bibitem{enriquemolinerheunentull:space}
P.~{Enrique Moliner}, C.~Heunen, and S.~Tull.
\newblock Space in monoidal categories.
\newblock {\em Quantum Phsyics and Logic}, 2017.
\newblock arXiv:1704.08086.

\bibitem{fell:bundles}
J.~M.~G. Fell.
\newblock {\em An extension of {M}ackey's method to {B}anach *-algebraic
  bundles}.
\newblock Number~90 in Memoirs. American Mathematical Society, 1969.

\bibitem{felldoran:bundles}
J.~M.~G. Fell and R.~S. Doran.
\newblock {\em Representations of *-algebras, locally compact groups, and
  {B}anach *-algebraic bundles}.
\newblock Number 125, 126 in Pure and Applied Mathematics. Academic Press,
  1988.

\bibitem{frank:monotonecomplete}
M.~Frank.
\newblock Hilbert {C}*-modules over monotone complete {C}*-algebras.
\newblock {\em Mathematische Nachrichten}, 175:61--83, 1995.

\bibitem{frankpaulsen:injective}
M.~Frank and V.~Paulsen.
\newblock Injective and projective {H}ilbert {C}*-modules, and {C}*-algebras of
  compact operators.
\newblock {\em arXiv:math.OA/0611348}, 2006.

\bibitem{furberjacobs:gelfand}
R.~W.~J. Furber and B.~P.~F. Jacobs.
\newblock From {K}leisli categories to commutative {C}*-algebras: probabilistic
  {G}elfand duality.
\newblock {\em Logical Methods in Computer Science}, 11(2):5, 2015.

\bibitem{ghezlimaroberts:wstarcategories}
P.~Ghez, R.~Lima, and J.~E. Roberts.
\newblock {$W^*$}-categories.
\newblock {\em Pacific Journal of Mathematics}, 120:79--109, 1985.

\bibitem{gogiosogenovese:nonstandard}
S.~Gogioso and F.~Genovese.
\newblock Infinite-dimensional categorical quantum mechanics.
\newblock In {\em Quantum Physics and Logic}, 2016.

\bibitem{hattori:stronglyseparable}
A.~Hattori.
\newblock On strongly separable algebras.
\newblock {\em Osaka Journal of Mathematics}, 2:369--372, 1965.

\bibitem{heunen:embedding}
C.~Heunen.
\newblock An embedding theorem for {H}ilbert categories.
\newblock {\em Theory and Applications of Categories}, 22(13):321--344, 2009.

\bibitem{heunenjacobs:kernels}
C.~Heunen and B.~Jacobs.
\newblock Quantum logic in dagger kernel categories.
\newblock 27(2):177--212, 2010.

\bibitem{heunenkarvonen:monads}
C.~Heunen and M.~Karvonen.
\newblock Monads on dagger categories.
\newblock {\em Theory and Applications of Categories}, 31(35):1016--1043, 2016.

\bibitem{heunenkissinger:cbh}
C.~Heunen and A.~Kissinger.
\newblock Can quantum theory be characterized in terms of information-theoretic
  constraints?
\newblock {\em arXiv:1604.05948}, 2016.

\bibitem{heunenkissingerselinger:cpproj}
C.~Heunen, A.~Kissinger, and P.~Selinger.
\newblock Completely positive projections and biproducts.
\newblock In {\em Quantum Physics and Logic {X}}, number 171 in Electronic
  Proceedings in Theoretical Computer Science, pages 71--83, 2014.

\bibitem{heunentull:regular}
C.~Heunen and S.~Tull.
\newblock Categories of relations as models of quantum theory.
\newblock In {\em Quantum Physics and Logic XII}, number 195 in Electronic
  Proceedings in Theoretical Computer Science, pages 247--261, 2015.

\bibitem{heunenvicary:cqm}
C.~Heunen and J.~Vicary.
\newblock {\em Categories for Quantum Theory: an introduction}.
\newblock Oxford University Press, 2017.

\bibitem{heunenvicarywester:two}
C.~Heunen, J.~Vicary, and L.~Wester.
\newblock Mixed quantum states in higher categories.
\newblock In {\em Quantum Physics and Logic}, volume 172 of {\em Electronic
  Proceedings in Theoretical Computer Science}, pages 304--315, 2014.

\bibitem{ivankov:quantization}
P.~Ivankov.
\newblock Quantization of noncompact coverings.
\newblock {\em arXiv:1702.07918}, 2017.

\bibitem{kellylaplaza:compactcategories}
G.~M. Kelly and M.~L. Laplaza.
\newblock Coherence for compact closed categories.
\newblock {\em Journal of Pure and Applied Algebra}, 19:193--213, 1980.

\bibitem{lance:hilbert}
E.~C. Lance.
\newblock {\em Hilbert {$C^*$}-modules}, volume 210 of {\em London Mathematical
  Society Lecture Note Series}.
\newblock Cambridge University Press, Cambridge, 1995.
\newblock A toolkit for operator algebraists.

\bibitem{paschke:selfdual}
W.~L. Paschke.
\newblock Inner product modules over {B}*-algebras.
\newblock {\em Transactions of the American Mathematical Society},
  182:443--468, 1973.

\bibitem{paulsen:positive}
V.~Paulsen.
\newblock {\em Completely bounded maps and operator algebras}.
\newblock Cambridge University Press, 2002.

\bibitem{pavlovtroitskii:branchedcoverings}
A.~A. Pavlov and E.~V. Troitskii.
\newblock Quantization of branched coverings.
\newblock {\em Russian Journal of Mathematical Physics}, 18(3):338--352, 2011.

\bibitem{pluta:corners}
R.~Pluta.
\newblock {\em Ranges of bimodule projections and conditional expectations}.
\newblock Cambridge Scholars, 2013.

\bibitem{raeburnwilliams:morita}
I.~Raeburn and D.~P. Williams.
\newblock {\em Morita equivalence and continuous-trace {C}*-algebras},
  volume~60 of {\em Mathematical Surveys and Monographs}.
\newblock American Mathematical Society, 1998.

\bibitem{rudin:analysis}
W.~Rudin.
\newblock {\em Real and complex analysis}.
\newblock McGraw-Hill, third edition, 1987.

\bibitem{selinger:graphicallanguages}
P.~Selinger.
\newblock A survey of graphical languages for monoidal categories.
\newblock In {\em New Structures for Physics}, Lecture Notes in Physics, pages
  289--355. Springer, 2009.

\bibitem{stormer:positive}
E.~St{\o}rmer.
\newblock {\em Positive linear maps of operator algebras}.
\newblock Springer, 2013.

\bibitem{takahashi:hilbertmodules}
A.~Takahashi.
\newblock Hilbert modules and their representation.
\newblock {\em Revista {C}olumbiana de matematicas}, 13:1--38, 1979.

\bibitem{taylor:measure}
M.~E. Taylor.
\newblock {\em Measure theory and integration}.
\newblock American Mathematical Society, 2006.

\bibitem{tomiyama:representation}
J.~Tomiyama.
\newblock Topological representation of {C}*-algebras.
\newblock {\em Tohoku Mathematical Journal}, 14(2):187--204, 1962.

\bibitem{tomiyamatakesaki:bundles}
J.~Tomiyama and M.~Takesaki.
\newblock Applications of fibre bundles to the certain class of {C}*-algebras.
\newblock {\em Tohoku Mathematical Journal}, 13(3):498--522, 1961.

\bibitem{tomiyama:conditionalexpectation}
Y.~Tomiyama.
\newblock On the projection of norm one in {W}*-algebras.
\newblock {\em Proceedings of the Japanese Academy}, 33(10):608--612, 1957.

\bibitem{vicary:quantumalgebras}
J.~Vicary.
\newblock Categorical formulation of quantum algebras.
\newblock {\em Communications in Mathematical Physics}, 304(3):765--796, 2011.

\bibitem{vicary:higher}
J.~Vicary.
\newblock Higher quantum theory.
\newblock {\em arXiv:1207.4563}, 2012.

\bibitem{weggeolsen:ktheory}
N.~E. {Wegge-Olsen}.
\newblock {\em K-theory and {C}*-algebras}.
\newblock Oxford University Press, 1993.

\bibitem{yamagami:duality}
S.~Yamagami.
\newblock Frobenius duality in {C}*-tensor categories.
\newblock {\em Journal of Operator Theory}, 52:3--20, 2004.

\bibitem{zito:cstarcategories}
P.~A. Zito.
\newblock 2-{C}*-categories with non-simple units.
\newblock {\em Advances in mathematics}, 210(1):122--164, 2007.

\end{thebibliography}

\end{document}